\documentclass[12pt]{amsart}

\usepackage{amssymb}
\usepackage{fancyhdr}
\usepackage{amsmath}
\usepackage{amsfonts}
\usepackage{mathrsfs}
\usepackage[all]{xy} % For commutative diagrams
\usepackage{color}

\newcommand{\disk}{\ensuremath{\mathbb{D}} } % unit disk
\newcommand{\riem}{\Sigma}  %Riemann surface
\theoremstyle{plain}
        \newtheorem{theorem}{Theorem}[section]
        \newtheorem{lemma}[theorem]{Lemma}
        \newtheorem{proposition}[theorem]{Proposition}
        \newtheorem{corollary}[theorem]{Corollary}

\theoremstyle{definition}
        \newtheorem{definition}[theorem]{Definition}

\theoremstyle{remark}
    \newtheorem{remark}[theorem]{Remark}

\numberwithin{equation}{section} % Equation labels are 'section'.'eq #'
\numberwithin{figure}{section} % Figures labela are 'section.'fig #'

\usepackage{fullpage}

\author{Eric Schippers}
\author{Wolfgang Staubach}

\begin{document}
\title{Plemelj-Sokhotski isomorphism for quasicircles in Riemann surfaces and the Schiffer operator}
\dedicatory{To the memory of our friend Peter C. Greiner}
% \title{The Schiffer kernel on Riemann surfaces bordered by quasicircles}
% \dedicatory{In memory of Menahem Max Schiffer}
\maketitle

\begin{abstract}
 Let $R$ be a compact Riemann surface and $\Gamma$ be a Jordan curve separating $R$ into connected components $\riem_1$ and $\riem_2$.  We consider Calder\'on-Zygmund type operators $T(\riem_1,\riem_k)$ taking the space of $L^2$ anti-holomorphic one-forms on $\riem_1$ to the space of $L^2$ holomorphic one-forms on $\riem_k$, which we call the Schiffer operators.  We extend results of Menahem M. Schiffer and others, which where confined to analytic Jordan curves $\Gamma$, to general quasicircles in a characterizing manner, and prove new identities for adjoints of the Schiffer operators.  Furthermore, we show that if $V$ is the space of anti-holomorphic one-forms orthogonal to $L^2$ forms on $R$ with respect to the inner product on $\riem_1$, then the Schiffer operator $T(\riem_1,\riem_2)$ is an isomorphism onto the set of exact one-forms on $\riem_2$.

 Using the relation between the Schiffer operator and a Cauchy-type integral involving Green's function, we also derive a jump decomposition (on arbitrary Riemann surfaces) for quasicircles and initial data which are boundary values of Dirichlet-bounded harmonic functions and satisfy the classical algebraic constraints.  In particular we show that the jump operator is an isomorphism on the subspace determined by these constraints.
\end{abstract}

\begin{section}{Introduction}
\begin{subsection}{Results and literature}

Let $\Gamma$ be a sufficiently regular curve separating a compact surface into two components $\riem_1$ and $\riem_2$. Given a sufficiently regular function $h$ on that curve, it is well known that there are holomorphic functions $h_k$ on $\riem_k$ such that
\[  h = h_2 - h_1 \]
if and only if $\int_\Gamma h \alpha =0$
for all holomorphic one forms on $R$.  In the plane, this is a consequence of the Plemelj-Sokhotski jump formula (which is a more precise formula in terms of a principal value integral).  The functions $h_k$ are obtained by integrating $h$ against the Cauchy kernel.

Different regularities of the curve and the function are possible.  In this paper, we show that the jump formula holds for quasicircles on compact Riemann surfaces, where the function $h$ is taken to be the boundary values of a harmonic function of bounded Dirichlet energy on either $\riem_1$ or $\riem_2$.  In the case that $\Gamma$ is analytic, this space agrees with the Sobolev $H^{1/2}$ space on $\Gamma$.  We showed in an earlier paper that the space of boundary values, for quasicircles, is the same for both $\riem_1$ and $\riem_2$, and the resulting map (which we call the transmission map) is bounded.

Since quasicircles are non-rectifiable, we replace the Cauchy integral by a limit of integrals along level curves of Green's function in $\riem_k$; for quasicircles, we show that this integral is the same whether one takes the limiting curves from within $\riem_1$ or $\riem_2$.  This relies on our transmission result mentioned above.  We also show that the map from the harmonic Dirichlet space $\mathcal{D}_{\mathrm{harm}}(\riem_k)$ to the direct sum of holomorphic Dirichlet spaces $\mathcal{D}(\riem_1) \oplus \mathcal{D}(\riem_2)$ obtained from the jump integral is an isomorphism.

We also consider a Calder\'on-Zygmund type integral operator on the space of one-forms which we call the Schiffer operator.   This was studied extensively by M. Schiffer and others in the plane and on Riemann surfaces (see Section \ref{se:attributions} for a discussion of the literature).  { Schiffer discovered deep relations between inequalities in function theory, potential theory and Fredholm eigenvalues, and properties of these operators.}   We extend many known results from analytic boundary to quasicircles.  We also derive some new identities for the adjoint of the Schiffer operator, and a complete set of identities relating the Schiffer operator to the jump integral in higher genus.   The derivative of the jump integral, when restricted to a finite co-dimension space of one-forms, equals the Schiffer operator.  We show that the restriction of the Schiffer operator to this finite co-dimensional space is an isomorphism.\\

In the case of simply-connected domains in the plane (where the finite co-dimensional space is the full space of one-forms), the fact that the Schiffer operator is an isomorphism is due to V. V. Napalkov and R. S. Yulmukhametov \cite{Nap_Yulm}.  In fact, they showed that it is an isomorphism precisely for domains bounded by quasicircles.  This is closely related to a result of Y. Shen \cite{ShenFaber}, who showed that the Faber operator of approximation theory is an isomorphism precisely for domains bounded by quasicircles.  Indeed, using Shen's result, the authors (at the time unaware of Napalkov and Yulmukhametov's result) derived a proof that the jump isomorphism and the Schiffer operator are isomorphisms precisely for quasicircles \cite{Schippers_Staubach_scattering}.  As mentioned above, here we generalize the jump and Schiffer isomorphism to Riemann surfaces separated by quasicircles.  We conjecture that the converse holds, as in the planar case; namely, if either of these is an isomorphism, then the separating curve is a quasicircle. \\

 We conclude with a few remarks on technical issues and related literature.
 The main hindrance to the solution of the Riemann boundary problem and the establishment of the jump decomposition is that quasicircles are highly irregular, and are not in general rectifiable.
 Riemann-Hilbert problems on non-rectifiable curves have been studied extensively by B. Kats, see e.g. \cite{Kats_2017} for the case of H\"older continuous boundary values, and the survey article \cite{Kats_survey} and references therein.  However the boundary values of Dirichlet bounded harmonic functions need not be H\"older continuous.  For Dirichlet spaces boundary values exist for quasicircles and the jump formula can be expressed in terms of certain limiting integrals.  A key tool here is our proof of the existence and boundedness of a transmission operator for harmonic functions in quasicircles \cite{Schippers_Staubach_general_transmission} (which, in the plane, also characterizes quasicircles \cite{Schippers_Staubach_transmission_sphere}).
 Indeed our approach to proving surjectivity of the Schiffer operator relies on the equality of the limiting integral from both sides.  { We have also found that the transmission operator has a clarifying effect on the theory as a whole.}

In this paper, approximation by functions which are analytic or harmonic on a neighbourhood of the closure plays an important role.  We rely on an approximation result of N. Askaripour and T. Barron \cite{AskBar} for $L^2$ $k$-differentials  on nested surfaces. Their result replaces the density of polynomials in the Bergman space of a Carath\'eodory domain used in proving the results in the case of the sphere. \\

\end{subsection}
\begin{subsection}{Outline of the paper}
  In Section \ref{se:preliminaries} we establish notation and state preliminary results.  We also outline previous results of the authors which are necessary here.   In Section \ref{Sec:Schiffer's comparison operators} we define the Schiffer operators, generalize known results to quasicircles, and establish some new identities.  In Section \ref{se:jump}, we give identities relating the Schiffer operator to a Cauchy-type integral (in general genus), we relate it to the jump decomposition, and establish the isomorphism theorems for the jump decomposition and the Schiffer operator.
\end{subsection}
\end{section}
\begin{section}{Notations and Preliminaries} \label{se:preliminaries}
\begin{subsection}{Forms and functions}

 We begin by establishing notation and terminology.

 Let $R$ be a Riemann surface, which we will always assume to be connected.
 For smooth real one-forms, define the dual of the almost
 complex structure  $\ast$  by

 \[  \ast (a\, dx + b \, dy) = a \,dy - b \,dx \]
 in a local holomorphic coordinate  $z=x+iy$.
This is independent of the choice of coordinates.  Harmonic functions $f$ on $R$
 are those which satisfy $d \ast d f =0$, while harmonic one-forms $\alpha$ are those which satisfy
 both $d\alpha =0$ and $\ast d \alpha =0$.  Equivalently, harmonic one-forms are those which can be expressed locally as $df$ for some harmonic function $f$.
  We consider complex-valued functions and forms.   Denote complex conjugation of functions and forms with an overline, e.g. $\overline{\alpha}$.

 Harmonic one-forms $\alpha$ can always be decomposed as a sum of a holomorphic and anti-holomorphic one-form.  The decomposition is unique.  On the other hand, harmonic functions do not possess such a decomposition.

The space of complex one-forms on $R$ has the natural inner-product
\[ (\omega_1,\omega_2)_{A(R)} = \frac{1}{2} \iint_R \omega_1 \wedge \ast \overline{\omega_2}; \]
Denote by $L^2(R)$ the set of one-forms which are $L^2$ with respect to this inner product.
The Bergman space of holomorphic one forms is
\[  A(R) = \{ \alpha \in L^2(R) \,:\, \alpha \ \text{holomorphic} \} \]
and the set of antiholomorphic $L^2$ one-forms will be denoted by $\overline{A(R)}$.  This notation is of course consistent, because $\beta \in \overline{A(R)}$ if and only if $\beta= \overline{\alpha}$ for some $\alpha \in A(R)$.  We will also denote
\[  A_{\text{harm}}(R) =\{ \alpha \in L^2(R) \,:\, \alpha \ \text{harmonic} \}. \]
Observe that $A(R)$ and $\overline{A(R)}$ are orthogonal with respect to the aforementioned inner product.

If $F: R_1 \rightarrow R_2$ is a conformal map, then we denote the pull-back of $\alpha \in A_{\text{harm}}(R)$
under $F$ by $F^*\alpha.$
%Explicitly, if $\alpha$ is given in local coordinates $w$ by $a(z) dz + \overline{b(z)} d\bar{z}$ and $z=f(w),$ then the pull-back is given by
%\[   f^* \left( a(w)\, dw + \overline{b(w)} \,d\bar{w} \right)= a(f(z)) f'(z)\, dz + \overline{b(f(z))} \overline{f'(z)}\, d\bar{z}.   \]

We also define the Dirichlet spaces by
\begin{align*}
   \mathcal{D}_{\text{harm}} (R)& = \{ f:R \rightarrow \mathbb{C} \,:\,
   df\in L^2 (R)\,\,\,\mathrm{and}\,\, \, d\ast df =0 \},\\
   \mathcal{D}(R)& = \{ f:R \rightarrow \mathbb{C} \,:\,
    df \in A(R) \}, \ \text{and} \\
    \overline{\mathcal{D}(R)} & = \{ f:R \rightarrow \mathbb{C} \,:\,
    df \in \overline{A(R)} \}. \\
\end{align*}

We can define a degenerate inner product on $\mathcal{D}_{\text{harm}}(R)$ by
\[   (f,g)_{\mathcal{D}_{\text{harm}}(R)} = (df,dg)_{A_{\text{harm}}(R)}.   \]

If we denote
\[  \mathcal{D}_{\text{harm}}(R)_q = \{ f \in \mathcal{D}_{\text{harm}}(R) \, :\, f(q)=0 \}  \]
for some $q \in R$, then the scalar product defined above is a genuine inner product on $\mathcal{D}_{\text{harm}}(R)_{q}$
and also makes it a Hilbert space. In what follows, a subscript $q$ on a space of functions indicates the subspace of functions such that $f(q)=0$.\\

If we now define the Wirtinger operators via their local coordinate expressions
\[   \partial f = \frac{\partial f}{\partial z}\, dz,  \ \ \
   \overline{\partial} f =  \frac{\partial f}{\partial \bar{z}}\, d \bar{z}, \]
then the aforementioned inner product can be written as
\begin{equation} \label{eq:inner_product_with_dbar_and_d}
 (f,g)_{\mathcal{D}_{\text{harm}}}(R) = \frac{i}{2} \iint_{R} \left[ \partial f \wedge \overline{\partial  g} -  \overline{\partial} f \wedge
 \partial \overline{g} \right].
\end{equation}
One can easily see from (\ref{eq:inner_product_with_dbar_and_d}) that $\mathcal{D}(R)$ and $\overline{\mathcal{D}(R)}$ are orthogonal with respect to the inner product. We also note that if $R$ is a planar domain and $f \in \mathcal{D}(R)$,  then $(f,f)_{\mathcal{D}(R)} = \iint_R |f'(z)|^2 dA$ where $dA$ denotes Lebesgue measure in the plane.\\

%We will need another useful expression for the norm.  It is easily checked that
%{\color{blue} check this next formula one last time}
%\begin{equation}  \label{eq:real_imaginary_decomp_norm}
%    \| H \|^2_{\mathcal{D}(R)_{\text{harm}}} = \| \text{Re}\, H \|^2_{\mathcal{D}(R)_{\text{harm}}} + \| \text{Im} \, H \|^2_{\mathcal{D}(R)_{\text{harm}}}.
%\end{equation}

Finally, we will repeatedly use the following elementary fact.
\begin{lemma} \label{th:uniform_controlled_by_L^2} Let $U \subset \mathbb{C}$ be an open set.  For any
 compact subset $K$ of $U$, there is a constant $M_K$ such that
 \[  \sup_{z\in K}|\alpha(z)| \leq M_K \| \alpha(z)\, dz \|_{A_{\mathrm{harm}}(U)}  \]
 for all $\alpha(z)\,dz \in A_{\mathrm{harm}}(U)$.

For any Riemann surface $R$, compact subset $K$ of $R$, and fixed $q \in R$, there is a constant $M_K$ such that
 \[\sup_{z\in K} |h(z)| \leq M_K \| h \|_{\mathcal{D}_{\mathrm{harm}}(R)_q}   \]
 for all $h \in \mathcal{D}_{\mathrm{harm}}(R)_q$.
\end{lemma}
The first claim is classical and the second claim is an elementary consequence of the first.

\end{subsection}
\begin{subsection}{Transmission of harmonic functions through quasicircles}
 In this section we summarize some necessary results of the authors. The proofs
 can be found in \cite{Schippers_Staubach_general_transmission}.
\\

 Let $R$ be a compact Riemann surface. Let $\Gamma$ be a Jordan curve in $R$, that is a homeomorphic image of $\mathbb{S}^1$.  We say that $U$ is a doubly-connected neighbourhood of $\Gamma$ if $U$ is an open set containing $\Gamma$, which is bounded by two non-intersecting Jordan curves each of which is homotopic to $\Gamma$ within the closure of $U$.  We say that a Jordan curve $\Gamma$ \emph{is strip-cutting} if there is a doubly-connected neighbourhood $U$ of $\Gamma$ and a conformal map $\phi:U \rightarrow \mathbb{A} \subseteq \mathbb{C}$ so that $\phi(\Gamma)$ is a Jordan curve in $\mathbb{C}$.
 We say that $\Gamma$ is a \emph{quasicircle} if $\phi(\Gamma)$ is a quasicircle.  In particular a quasicircle is a strip-cutting Jordan curve.   A closed analytic curve is strip-cutting by definition.

If $R$ is a Riemann surface and $\riem \subset R$ is a proper subset of $R$, then we say
that $g(w,z)$ is the Green's function for $\riem$ if $g(w,\cdot)$ is harmonic on $R \backslash \{w\}$,  $g(w,z) + \log{|\phi(z)-\phi(w)|}$ is harmonic in $z$
for a local parameter $\phi:U \rightarrow \mathbb{C}$ in an open neighbourhood $U$ of $w$, and $\lim_{z \rightarrow z_0} g(w,z) =0$ for all $z_0 \in \partial \riem$ and $w \in \riem$.  Green's function is unique and symmetric, provided that it exists.  In this paper, we will consider only the
case where $R$ is compact and no boundary component of $\riem$ reduces to a point, so Green's function of $\riem$ exists; see for example L. Ahlfors and L. Sario \cite[II.3]{Ahlfors_Sario}.

Now let $\riem$ be one of the connected components in $R$ of the complement of $\Gamma$.
Fix a point $q \in \riem$ and let $g_q$ be Green's function of $\riem$ with singularity at $q$.
We associate to $g_q$ a biholomorphism from a doubly-connected region in $\riem$, one of whose borders is $\Gamma$, onto an annulus as follows.
Let $\gamma$ be a smooth curve in $\riem$ which is homotopic to $\Gamma$, and let
$m = \int_\gamma \ast dg_q$.  If $\tilde{g}$ denotes the multi-valued harmonic
conjugate of $g_q$, then the function
\[  \phi = \exp{[-2 \pi  (g_q + i \tilde{g})/m]}  \]
is holomorphic and single-valued on some region $A_r$ bounded by $\Gamma$ and a level curve
$\Gamma^{q}_r = \{ z\,:\, g_q(z) = r \}$ of $g_q$ for some $r>0$.
A standard use of the argument principle shows that $\phi$ is one-to-one
and onto the annulus $\{ z : e^{-2\pi r/m} <|z|<1 \}$.  It can be shown that $\phi$ has a continuous extension to  $\Gamma$ which is a homeomorphism of $\Gamma$ onto $\mathbb{S}^1$.  By decreasing $r$, one can also arrange that $\phi$ extends analytically to a
neighbourhood of $\Gamma^{q}_r$.

We call this the \textit{canonical collar chart} with respect to $(\riem,q)$.
It is uniquely determined up to a rotation and the choice of $r$ in the definition of domain.

We say that a closed set $I \subseteq \Gamma$ is null with respect to $(\riem,q)$ if $\phi(I)$ has logarithmic capacity zero in $\mathbb{S}^1$.
The notion of a null set does not depend on the position of the singularity $q$.  For quasicircles, it is also independent of the side of the curve.
\begin{theorem}  Let $R$ be a compact Riemann surface and $\Gamma$ be a strip-cutting Jordan-curve
 separating $R$ into two connected components $\riem_1$ and $\riem_2$.  Let $I$ be a closed set in $\Gamma$.
 \begin{enumerate}
  \item $I$ is null with respect to $(\riem_1,q)$ for some $q \in \riem_1$ if and only if it is null with respect to $(\riem_1,q)$ for all $q \in \riem_1$.
  \item If $\Gamma$ is a quasicircle, then $I$ is null with respect to $(\riem_1,q)$ for some $q \in \riem_1$ if and only if $I$ is null with respect to $(\riem_2,p)$ for all $p \in \riem_2$.
 \end{enumerate}
\end{theorem}
Thus for quasicircles we can say ``$I$ is null in $\Gamma$'' without ambiguity.  For strip-cutting Jordan curves, we may say that ``$I$ is null in $\Gamma$ with respect to $\riem$'' without ambiguity.

\begin{definition}
Given a function $f$ on an open neighbourhood of $\Gamma$ in the closure of $\riem$, we say that the limit of $f$ exists conformally non-tangentially at $p \in \Gamma$ with respect to $(\riem,q)$ if  $f\circ \phi^{-1}$ has non-tangential limits at $\phi(p)$ where $\phi$ is the canonical collar chart induced by Green's function $g_q$ of $\riem$.  The conformal non-tangential limit of $f$ at $p$ is defined to be the non-tangential limit of $f \circ \phi^{-1}$.
\end{definition}
We will abbreviate ``conformally non-tangential'' as $\mathrm{CNT}$ throughout the paper.

\begin{theorem}   Let $R$ be a compact Riemann surface and let $\Gamma$ be a strip-cutting Jordan curve separating $R$ into two connected components.  Let $\riem$ be one of these components. For any $H \in \mathcal{D}_{\mathrm{harm}}(\riem)$, the $\mathrm{CNT}$ limit of $H$ exists at every point in $\Gamma$ except possibly on a null set with respect to $\riem$.  For any $q$ and $q'$ in $\riem$, the boundary values so obtained agree except on a null set $I$ in $\Gamma$.   If $H_1,H_2 \in \mathcal{D}(\riem)$ have the same $\mathrm{CNT}$ boundary values except
on a null set then $H_1 = H_2$.
\end{theorem}

From now on, the terms ``$\mathrm{CNT}$ boundary values'' and ``boundary values'' of a Dirichlet-bounded harmonic function refer to the $\mathrm{CNT}$ limits thus defined except possibly on a null set.  Also, if $\Gamma$ is a quasicircle,  we say that two functions $h_1$ and $h_2$ agree on $\Gamma$ ($h_1=h_2$) if they agree except on a null set.   Outside of this section we will drop the phrase ``except on a null set'', although it is implicit wherever boundary values are considered.

The set of boundary values of Dirichlet-bounded harmonic functions in a certain sense determined only by a neighbourhood of the boundary.  For quasicircles, it is side-independent: that is, the set of boundary values of the Dirichlet spaces of $\riem_1$
and $\riem_2$ agree.

  To make the first statement precise we define a kind of one-sided neighbourhood of $\Gamma$ which we clal a  collar neighourhood.  Let $\Gamma$ be a strip-cutting Jordan curve in a Riemann surface $R$.  By a collar neighbourhood of $\Gamma$ we mean an open set $A$, bounded by two Jordan curves one of which is $\Gamma$, and such that  (1) the other Jordan curve $\Gamma'$ is homotopic to $\Gamma$ in the closure of $A$ and (2) $\Gamma' \cap \Gamma$ is empty.   For example, if $U$ is a doubly-connected neighbourhood of $\Gamma$, and $\Gamma$ separates a compact Riemann surface $R$ into two connected components, the intersection of $U$ with one of the the components is a collar neighbourhood.  Also, the domain of the canonical collar chart is a collar neighbourhood
  if the annulus $r<|z|<1$ is chosen with $r$ sufficiently close to one.

\begin{theorem}  \label{th:transmission_equivalences} Let $R$ be a compact Riemann surface and let $\Gamma$ be a strip-cutting Jordan curve
separating $R$ into connected components $\riem_1$ and $\riem_2$.  Let $h$ be a function defined on $\Gamma$, except possibly on a null set in $\Gamma$.  The following are equivalent.
\begin{enumerate}
 \item[$(1)$] There is some $H \in \mathcal{D}_{\mathrm{harm}}(\riem_1)$ whose $\mathrm{CNT}$ boundary values agree with $h$ except possibly on a null set.
 \item[$(2)$] There is a collar neighbourhood $A$ of $\Gamma$ in $\riem_1$, one of whose boundary components is $\Gamma$, and
  some $H \in \mathcal{D}(A)$ whose $\mathrm{CNT}$ boundary values agree with $h$ except
  possibly on a null set with respect to $\riem_1$.
\end{enumerate}
 If $\Gamma$ is a quasicircle, then the following may be added to the list of equivalences above.
\begin{enumerate}   \setcounter{enumi}{2}
 \item[$(3)$] There is some $H \in \mathcal{D}_{\mathrm{harm}}(\riem_2)$ whose $\mathrm{CNT}$ boundary values agree with $h$ except possibly on a null set.
 \item[$(4)$] There is a collar neighbourhood $A$ of $\Gamma$ in $\riem_2$, one of whose boundary components is $\Gamma$, and
  some $H \in \mathcal{D}(A)$ whose $\mathrm{CNT}$ boundary values agree with $h$ except possibly on a null set.
\end{enumerate}
\end{theorem}

Thus, for a quasicircle $\Gamma$ we may define $\mathcal{H}(\Gamma)$ to be the set of equivalence classes of functions $h:\Gamma \rightarrow \mathbb{C}$ which are boundary values of elements of $\mathcal{D}_{\text{harm}}(\riem_1)$ except possibly on a null set, where we define two such functions to be equivalent
if they agree except possibly on a null set.

This theorem also induces a map from $\mathcal{D}_{\text{harm}}(\riem_1)$ to $\mathcal{D}_{\text{harm}}(\riem_2)$ as follows:
\begin{definition}  Let $\Gamma$ be a quasicircle in a compact Riemann surface $R$, separating it into two connected components $\riem_1$ and $\riem_2$.
Given $H \in \mathcal{D}(\riem_1)$, let $h$ be the $\mathrm{CNT}$ boundary values of $H$ on $\Gamma$.  Define $\mathfrak{O}(\riem_1,\riem_2) H$  \footnote{The notation $\mathfrak{O}$ for this transmission operator stems from the first letter in the Old English word ``oferferian'' which means ``to transmit'' (or ``to overfare'').}
to be the unique element of $\mathcal{D}_{\mathrm{harm}}(\riem_2)$ with boundary values equal to $h$.
\end{definition}
This operator enables one to transmit harmonic functions from one side of the Riemann surface to the other side through the quasicircle $\Gamma$.

\begin{theorem} \label{th:transmission_bounded} Let $R$ be a compact Riemann surface and
 $\Gamma$ be a quasicircle separating $R$ into components $\riem_1$ and $\riem_2$.  The map
 \[ \mathfrak{O}(\riem_1,\riem_2):\mathcal{D}_{\mathrm{harm}}(\riem_1)
  \rightarrow \mathcal{D}_{\mathrm{harm}}(\riem_2)  \] induced by $\mathrm{Theorem\,\,\, \ref{th:transmission_equivalences}}$ is bounded with respect to the Dirichlet semi-norm.
\end{theorem}

\end{subsection}
\end{section}
\begin{section}{Schiffer's comparison operators}\label{Sec:Schiffer's comparison operators}
\begin{subsection}{Assumptions}
 The following notation and assumptions will be in place throughout the rest of the paper (see the relevant sections for further explanations):
 \begin{itemize}
  \item $R$ is a compact Riemann surface;
  \item $\Gamma$ is a strip-cutting Jordan-like curve separating $R$;
  \item $\riem_1$ and $\riem_2$ are the connected components of $R \backslash \Gamma$;
  \item $\riem$ stands for an unspecified component $\riem_1$ or $\riem_2$;
  \item $\Gamma$ is positively oriented with respect to $\riem_1$;
  \item $\Gamma^{p_k}_\epsilon$ the level curves of Green's function $g_{\riem_k}(\cdot,p_k)$ with respect to some fixed points $p_k \in \riem_k$;
  \item when an integrand depends on two variables, we will use the notation $\iint_{\riem,w}$ to
   specify that the integration takes place over the variable $w$.
 \end{itemize}

  We will sometimes alter the assumptions or repeat them for emphasis.  When no assumptions are indicated at all, the above assumptions are in place.
\end{subsection}
\begin{subsection}{Schiffer's comparison operators: definitions}
 Following for example Royden \cite{Royden}, we define Green's function of $R$ to be the unique function
 $g(w,w_0;z,q)$ such that
 \begin{enumerate}
  \item $g$ is harmonic in $w$ on $R \backslash \{z,q\}$;
  \item for a local coordinate $\phi$ on an open set $U$ containing $z$, $g(w,w_0;z,q) + \log| \phi(w) -\phi(z) |$ is harmonic
   for $w \in U$;
  \item for a local coordinate $\phi$ on an open set $U$ containing $q$, $g(w,w_0;z,q) - \log| \phi(w) -\phi(z) |$ is harmonic
   for $w \in U$;
  \item $g(w_0,w_0;z,q)=0$ for all $z,q,w_0$.
 \end{enumerate}
 It can be shown that $g$ exists, is uniquely determined by these properties, and furthermore
 satisfies the symmetry properties
 \begin{align}
  g(w,w_1;z,q) & = g(w,w_0;z,q) - g(w_1,w_0;z,q) \label{eq:g_w_0_dependence} \\
  g(w_0,w;z,q) & = - g(w,w_0;z,q) \label{eq:g_interchange_w} \\
  g(z,q;w,w_0) & = g(w,w_0;z,q).  \label{eq:g_interchange_both}
 \end{align}
 In particular, $g$ is also harmonic in $z$ away from the poles.

 We will treat $w_0$ as fixed throughout the paper, and notationally drop the dependence
 on $w_0$ as much as possible.
 In fact, it follows immediately from (\ref{eq:g_w_0_dependence}) that $\partial_w g$ is independent
 of $w_0$.  All formulas of consequence in this paper are independent of $w_0$ for this reason.

%We say that a continuous curve $\Gamma$ in a compact Riemann surface $R$ is Jordan-like if there is an open
% neighbourhood $U$ of $\Gamma$ and a biholomorphism $\phi:U \rightarrow B \subseteq \mathbb{C}$ such
% that $\phi(\Gamma)$ is a Jordan-curve.  We say that $\Gamma$ separates $R$ if the complement $R \backslash \Gamma$ has two connected components.

% We will frequently use the following notation.  Let $\riem$ be one of the components of the complement of $R \backslash \Gamma$.  Let $g_\riem$ be Green's function of $\riem$ and let $\Gamma^p_\epsilon$
% denote the level curves $g(w,p) = \epsilon$ for fixed $p \in \riem$.

The following is an immediate consequence of the residue theorem and the fact that $g$ is harmonic in $w$.
\begin{theorem} \label{th:Greens_is_Cauchy_kernel}
  Let $\Gamma$ be a closed analytic curve separating $R$, enclosing $\riem$, which is positively
  oriented with respect to $\riem$.
  If $h$ is holomorphic on $\riem$, and $z,q \notin \Gamma$, then
  for any fixed $p \in \riem$
  \[  - \lim_{\epsilon \searrow 0} \frac{1}{\pi i} \int_{\Gamma_\epsilon^p} h(w)\, \partial_w g(w,w_0;z,q) = \chi_\riem (z) h(z) - \chi_\riem (q) h(q)   \]
  where $\chi_\riem$ is the characteristic function of $\riem$.
\end{theorem}

  We will also need the following well-known reproducing formula for Green's function of $\riem$.
 \begin{theorem}  \label{th:Greens_reproducing}
 Let $R$ be a compact Riemann surface and $\Gamma$ be a
  strip-cutting Jordan curve separating $R$.  Let $\riem$ be one of the components of
  the complement of $\Gamma$.  For any $h \in \mathcal{D}_{\mathrm{harm}}(\riem)$,
  we have
  \[   h(z) = \lim_{\epsilon \searrow 0}  - \frac{1}{\pi i} \int_{\Gamma^p_\epsilon}
   \partial_w g_{\riem}(w,z) \, h(w).  \]
 \end{theorem}

Next we turn to the definitions of the relevant kernel forms. Let $R$ be a compact Riemann surface, and let $g(w,w_0;z,q)$ be the Green's function.  We define the Schiffer kernel to be the bi-differential
 \[   L_R(z,w) = - \frac{1}{\pi i} \partial_z \partial_w g(w,w_0;z,q).    \]
 We also define the Bergman  kernel function
 \[  K_R (z,w) = - \frac{1}{\pi i} \partial_z \overline{\partial}_{{w}} g(w,w_0;z,q).    \]

%\begin{remark} By a bi-differential we mean an expression written in local coordinates
%$\zeta$, $\eta$ for example in the form $h(\zeta,\eta) \,d\zeta \,d\eta$ or $k(\zeta,\eta)  \,d\zeta,d\eta$.  We give a formal definition for those readers who desire one; those who do not should feel free to skip this remark.

%Let $\pi: E \rightarrow M \times M$ be the vector bundle of
%real multilinear maps $\alpha:T_z R \times T_w R \rightarrow \mathbb{C}$.  A bi-differential
%is a smooth section of $E$.  We say that $\alpha$ is holomorphic in $z$ (say) if for fixed $w$
%and $v \in T_w R$ we have that $\alpha(\cdot,v)$ is a holomorphic differential in $z$.  In
%particular, for fixed $z$ it satisfies $\alpha(J_z v_1,v_2)= i \alpha(v_1,v_2)$
% where $J_z : T_z R \rightarrow T_z R$ is the almost complex structure on $R$.
% One similarly defines anti-holomorphicity
% in $z$, holomorphicity in $w$, and anti-holomorphicity in $w$.
% It is then easily checked, for example, that a holomorphic differential in $z$ and $w$ is expressed in local coordinates $(\zeta(z),\eta(w))$ on some neighbourhood in $R \times R$ by $h(\zeta,\eta) \,d\zeta \, d\eta$ for some function $h$ which is holomorphic in $\zeta$ and $\eta$.

% The reader can also phrase this in terms of the complexified cotangent bundle and its splitting if desired.  The definition here simplifies some geometric statements (e.g. equation (\ref{eq:level_curve_identity}) below).  {\color{blue} I think we might want to remove this, and just
% refer to Ahlfors and Sario, as well as Schiffer and Spencer.  Most people think they get this.}
% \end{remark}

 For non-compact surfaces $\riem$ with border, with Green's function $g$, we define
 \[  L_\riem(z,w) =  - \frac{1}{\pi i} \partial_z \partial_w g(w,z).  \]
 and
 \[  K_\riem (z,w) = -\frac{1}{\pi i} \partial_z \overline{\partial}_{{w}} g(w,z).  \]
Then the following identity holds.   For any vector $v$ tangent to $\Gamma^{w}_{\epsilon}$ at a point $z$, we have
  \begin{equation} \label{eq:level_curve_identity}
   \overline{K_\riem(z,w)}(\cdot,v) = -L_\riem(z,w)(\cdot,v)
  \end{equation}
 This follows directly from the fact that the one form $\partial_z g (z,w) + \overline{\partial}_{{z}} g(z,w)$
 vanishes on tangent vectors to the level curve $\Gamma^{w}_{\epsilon}$.

  It is well known that for all $h \in A(\riem)$
 \begin{equation}  \label{eq:Bergman_reproducing}
  \iint_\riem K_\riem(z,w) \wedge h(w) = h(z).
 \end{equation}
 For compact surfaces, the reproducing property of the Bergman kernel is established in \cite{Royden}.

\begin{proposition}  \label{pr:bunch_o_identities}
 Let $R$ be a compact Riemann surface with Green's function $g(w,w_0;z,q)$.  Then
 \begin{enumerate}
  \item[$(1)$] $L_R$ and $K_R$ are independent of $q$ and $w_0$.
  \item[$(2)$] $K_R$ is holomorphic in $z$ for fixed $w$, and anti-holomorphic in $w$ for fixed $z$.
  \item[$(3)$] $L_R$ is holomorphic in $w$ and $z$, except for a pole of order two when $w=z$.
\item[$(4)$] $L_R(z,w)=L_R(w,z)$.
  \item[$(5)$] $K_R(w,z)=\overline{K_R(z,w)}$.
 \end{enumerate}
 For non-compact Riemann surfaces $\riem$ with Green's function, $(2)-(5)$ hold with $L_R$ and $K_R$ replaced by $L_\riem$ and $K_\riem$.
\end{proposition}
\begin{remark}
 The symmetry statements (4) and (5) are formally expressed
 as follows.  If $D:R \times R \rightarrow R \times R$ is the map $D(z,w)=(w,z)$ then
 $D^*L =L \circ D$ and $D^*K =\overline{K \circ D}$.
\end{remark}
\begin{proof} It follows immediately from (\ref{eq:g_w_0_dependence})
 that
 \[ \partial_w g(w,w_1;z,q) = \partial_w g(w,w_0;z,q) \ \ \text{and} \ \ \partial_{\bar{w}} g(w,w_1;z,q) = \partial_{\bar{w}} g(w,w_0;z,q),  \]
 so $L_R$ and $K_R$ are independent of $w_0$.  Applying
 (\ref{eq:g_interchange_both}) shows that similarly $\partial_w g$ and $\partial_{\bar{w}}$
 are independent of $q$, and hence the same holds for $L_R$ and $K_R$.  This demonstrates
 that property (1) holds.

 Since $g$ is harmonic in $w$, $\partial_{w} \overline{\partial}_{{w}} g(w,w_0;z,q) =0$ so $K_R$ is
 anti-holomorphic in $w$.  As observed above, (\ref{eq:g_interchange_w}) shows that $g$ is also harmonic in $z$, so we similarly have that $K_R$ is holomorphic in $z$.  This demonstrates (2).

 Similarly harmonicity of $g$ in $z$ and $w$ implies that $L_R$ is holomorphic in $z$ and $w$.
 The fact that $L_R$ has a pole of order two at $z$ follows from the fact that $g$ has a logarithmic
 singularity at $w=z$.  This proves (3).

 Properties (4) and (5) follow from equation (\ref{eq:g_interchange_both})
 applied directly to the definitions of $L_R$ and $K_R$.

 The non-compact case follows similarly from the harmonicity with logarithmic singularity of $g_\riem$, and the symmetry $g_\riem(z,w) = g_\riem(w,z)$
\end{proof}
One can find the constant at the pole of $L$ from the definition.  Expressed in a local holomorphic coordinates $\eta = \phi(w)$ near a fixed point $\zeta = \phi(z)$,
\begin{equation}  \label{eq:L_expressed_locally}
 (\phi^{-1} \times \phi^{-1})^* L(z,w) = \left(-\frac{1}{2\pi i} \frac{1}{(\zeta-\eta)^2}  + H(\eta) \right) d\zeta d\eta
\end{equation}
where $H(\eta)$ is holomorphic in a neighbourhood of $\zeta$.  In most sources \cite{BergmanSchiffer}, the integral kernel is expressed as a function (rather than a form) to be integrated against the Euclidean area form $dA_\eta = d\bar{\eta} \wedge d{\eta}/{2i}$.  E.g. if $\overline{\alpha(w}$ is a holomorphic one-form given in local coordinates by $(\phi^{-1})^* \alpha (\eta)= \overline{f(\eta)} d{\overline{\eta}}$ then we obtain the local expression
\[ (\phi^{-1} \times \phi^{-1})^{*} L(z,w) \wedge_\eta \phi^{-1*}\alpha(w)
    = \left( \frac{1}{\pi} \frac{1}{(\zeta-\eta)^2} + H(\eta) \right) \overline{f(\eta)}
     \, d\zeta \, dA_\eta \]
which agrees with the classical normalization \cite{BergmanSchiffer}.

 Now let $R$ be a compact Riemann surface and let $\Gamma$ be a strip-cutting Jordan curve.  Assume that $\Gamma$ separates $R$ into two surfaces $\riem_1$ and $\riem_2$.  We will mostly be concerned with the case that $\Gamma$ is a quasicircle.

 Let $A(\riem_1 \cup \riem_2)$ denote the set of one-forms on $\riem_1 \cup \riem_2$ which are
 holomorphic and square integrable.  Note that we do not require the existence of a holomorphic or continuous extension to the closure of $\riem_1 \cup \riem_2$.
 For $k=1,2$ define the restriction operators
 \begin{align*}
  \mathrm{Res}(\riem_k):A(R) & \rightarrow A(\riem_k) \\
  \alpha & \mapsto \left. \alpha \right|_{\riem_k}
 \end{align*}
 and
 \begin{align*}
  \text{Res}_0(\riem_k): A(\riem_1 \cup \riem_2) & \rightarrow A(\riem_k) \\
    \alpha & \mapsto \left. \alpha \right|_{\riem_k}.
 \end{align*}
 It is obvious that these are bounded operators.
\begin{definition}
For $k=1,2$, we define the Schiffer comparison operators
 \begin{align*}
  T(\riem_k): \overline{A(\riem_k)} & \rightarrow A(\riem_1 \cup \riem_2)  \\
  \overline{\alpha} & \mapsto \iint_{\riem_k} L_R(\cdot,w) \wedge \overline{\alpha(w)}.
 \end{align*}
\\
and
\begin{align*}
 S(\riem_k): A(\riem_k) & \rightarrow A(R) \\
  \alpha & \mapsto \iint_{\riem_k}  K_R(\cdot,w) \wedge \alpha(w).
\end{align*}
Also, we define for $j,k \in \{1,2\}$
 \[  T(\riem_j,\riem_k) = \text{Res}_0(\riem_k) T(\riem_j): \overline{A(\riem_j)} \rightarrow A(\riem_k).  \]
\end{definition}

Note that the operator $S$ is bounded and the image is clearly in
$A(R)$. Moreover, for $j \neq k$, the integral kernel of this operator is nonsingular, but if $j =k$, then the kernel is singular.  We will show below that the image is in fact in $A(\riem_k)$. \\

First we require an identity of Schiffer.  Although this identity was only stated for analytically bounded domains, it is easily seen to hold in greater generality.
 \begin{theorem}  \label{th:Schiffer_vanishing_identity}
  For all $\overline{\alpha} \in \overline{A(\riem)}$
  \[   \iint_{\riem,w}  L_\riem(z,w) \wedge \overline{\alpha(w)}=0. \]
 \end{theorem}
 \begin{proof}

  We assume momentarily that $\alpha$ has a holomorphic extension to the closure of $\riem$ and that $\Gamma$
  is an analytic curve.
 Let $z \in \riem$ be fixed but arbitrary, and choose a chart $\zeta$
  near $z$ such that $\zeta(z)=0$.  Write $\overline{\alpha}$ locally as $\overline{f(\zeta)} d\overline{\zeta}$
  for some holomorphic function $f$.  Let $C_r$ be the curve $|\zeta|=r$, and denote
  its image in $\riem$ by $\gamma_r$.  Fixing  $p \in \riem$ and using Stokes' theorem yield
  \begin{equation*}
   \iint_{\riem,w} L_{\riem}(z,w) \wedge \overline{\alpha(w)} =
   - \lim_{\epsilon \searrow 0}  \frac{1}{\pi i} \int_{\Gamma^p_\epsilon}
   \partial_z g(w,z) \overline{\alpha(w)} + \lim_{r \searrow 0}
   \frac{1}{\pi i} \int_{C_r}
    \partial_z g(w,z) \overline{\alpha(w)}.
  \end{equation*}
  The first term goes to zero uniformly as $\epsilon \rightarrow 0$.  Writing the
  second term in coordinates $\eta = \phi(w)$ in a neighbourhood of $\zeta$ for
  fixed $\zeta$ (see equation (\ref{eq:L_expressed_locally}))  we obtain
  \begin{align*}
   \iint_{\riem,w} L_{\riem}(z,w) \wedge \overline{\alpha(w)} & = \lim_{r \searrow 0}
    -\frac{1}{2 \pi i} \int_{C_r}
     \left(  \frac{1}{\eta} + h(\zeta) \right) \overline{f(\eta)} d\bar{\eta}
  \end{align*}
  where $h$ is some harmonic function in a neighbourhood of $0$. Now since both terms on the right hand side go to zero, we  obtain the desired result.

  Note that this shows that the principal value integral can be taken with respect
  to any local coordinate with the same result.  Furthermore, the integral is conformally invariant.  Thus, we may assume that $\riem$ is a subset of its
  double and $\Gamma$ is analytic. By \cite[Proposition 2.2]{AskBar}, the set of
  holomorphic one-forms on an open neighbourhood of the closure of $\riem$ is dense in $A(\riem)$.
  The $L^2$ boundedness of the $L_\riem$ operator yields the desired result.
 \end{proof}

 This implies that for $R$, $\Gamma$, and $\riem$ as in Theorem \ref{th:Schiffer_vanishing_identity}, we can write
 \begin{equation}  \label{eq:nonsingular_Schiffer}
  [T(\riem,\riem) \alpha](z)= \iint_{\riem,w} \left(L_R(z,w) - L_\riem(z,w) \right)
  \wedge \overline{\alpha(w)},  \end{equation}
 which has the advantage that the integral kernel is non-singular.
 \begin{remark}
  The above expression shows that the operator
  $T(\riem,\riem)$ is well-defined.  The subtlety is that the principal
  value integral might depend on the choice of coordinates, which determines
  the ball which one removes in the neighbourhood of the singularity. Since the
  integrand is not in $L^2$, different exhaustions of $\riem$ might in principle lead to different
  values of the integral.

  However the proof of Theorem \ref{th:Schiffer_vanishing_identity}
  shows that the integral of $L_\riem$ is independent of the choice of coordinate
  near the singularity.  Since the integrand of (\ref{eq:nonsingular_Schiffer})
  is $L^2$ bounded, it is independent of the choice of exhaustion; combining this
  with Theorem \ref{th:Schiffer_vanishing_identity} shows that the integral in the
  definition of $T(\riem,\riem)$ is independent of the choice of exhaustion.
  One may also obtain this fact from the general theory of Calder\'on-Zygmund operators on
  manifolds, see \cite{Seeley}.
 \end{remark}

 \begin{theorem}  \label{th:T_boundedness}
  Let $R$ be a compact Riemann surface, and $\Gamma$ be a strip-cutting Jordan curve in $R$.  Assume that $\Gamma$ separates $R$ into two surfaces $\riem_1$ and $\riem_2$.   Then $T(\riem_j) \overline{\alpha} \in A(\riem_1 \cup \riem_2)$ for all $\alpha \in A(\riem_j)$ for $j=1,2$.
  Furthermore for all $j,k \in \{1,2\}$,
  $T(\riem_j)$ and $T(\riem_j,\riem_k)$ are bounded operators.
 \end{theorem}
 \begin{proof} {  Fix $j$ and let $k \in \{ 1,2 \}$ be such that
 $k \neq j$.
 By  \eqref{eq:nonsingular_Schiffer}
 we observe that
 \begin{equation}\label{desingularisation of Tj}
     T(\riem_j) \overline{\alpha} (z) = \left\{
       \begin{array}{cc}
       \iint_{\riem_{j,w}}  L_R(z,w)
  \wedge \overline{\alpha(w)} & z \in \riem_k \\
       \iint_{\riem_{j,w}} \left(L_R(z,w) - L_{\riem_j}(z,w) \right)
  \wedge \overline{\alpha(w)} & z \in \riem_{j}  \end{array} \right.
 \end{equation}
 }
 The integrand in both terms \eqref{desingularisation of Tj} is non-singular and holomorphic in $z$ for each $w\in \riem_j$, and
 {furthermore both integrals are locally bounded in $z$}.
% \[  \iint_{\riem_j,w}| \left(L_R(z,w) - L_\riem_j(z,w) \right)
%  \wedge \overline{\alpha(w)}|  \]
 Therefore the holomorphicity of $T(\Sigma_j)\overline{\alpha}$ follows by moving the $\overline{\partial}$ inside \eqref{eq:nonsingular_Schiffer}, and using the holomorphicity of the integrand. This also implies the holomorphicity of $T(\riem_j,\riem_k).$

Regarding the boundedness, the operator $T(\Sigma_j)$ is defined by integration against the $L$-Kernel which in local coordinates is given by $\frac{1}{\pi(\zeta-\eta)^2}$, modulo a holomorphic function. Since the singular part of the kernel is a Calder\'on-Zygmund kernel we can use the theory of singular integral operators on general compact manifolds, developed by R. Seeley in \cite{Seeley} to conclude that that, the operators with kernels such as $L_R(z,w)$ are bounded on $L^p$ for $1<p<\infty$. The boundedness of $T(\riem_j, \riem_k)$ follows from this and the fact that $R_0(\riem_j)$ is also bounded.

 \end{proof}

% {\color{blue} Eric says: I moved the remark below}
%\begin{remark}
%The operators $T(\riem_i,\riem_j)$ somehow play the role of "wave operators" in the Schiffer-type scattering theory. {\color{blue} Wulf says: depending on the "twist" we intend to give to this paper, this remark will stay or will be removed. }
%\end{remark}
\end{subsection}
\begin{subsection}{Attributions} \label{se:attributions}
 The comparison operators $T(\riem_j,\riem_k)$ were studied extensively by Schiffer \cite{Schiffer_first}, \cite{DurenZalcman_Schiffer_collected_I,DurenZalcman_Schiffer_collected_II}, and also together with other authors, e.g. Bergman and Schiffer \cite{BergmanSchiffer}. In the setting of planar domains, a comprehensive outline of the theory was developed in a chapter in
\cite{Courant_Schiffer}. The comparison theory for Riemann surfaces
can be found in Schiffer and Spencer \cite{Schiffer_Spencer}.  See also our review paper \cite{SchippersStaubach_comparison}.

 In this section, we demonstrate some necessary identities for the Schiffer operator.  Most of the identities were stated by for example Bergman and Schiffer \cite{BergmanSchiffer}, Schiffer \cite{Courant_Schiffer}, and Schiffer and Spencer \cite{Schiffer_Spencer} for the case of analytic boundaries. Versions can be found in different settings, for example multiply-connected domains in the sphere, nested multiply-connected domains, and Riemann surfaces.

 On the other hand, we introduce here several identities involving the adjoints of the operators, which
 Schiffer seems not to have been aware of.  These are Theorem \ref{th:Bergman_comparison_restriction},
 Theorem \ref{th:T_adjoint}, and Theorem \ref{th:two_kernels_adjoint_identity}.  The introduction of the adjoint operators has significant clarifying power. Proofs of the remaining identities are included because it is necessary to show
 that they hold for regions bordered by quasicircles.

%  \begin{remark}
% %   It is not clear why
% %  Schiffer failed to consider the adjoint, but we speculate that it may have been prevented by his definition of the
% %  comparison operators as complex anti-linear operators.

%  Theorem \ref{th:two_kernels_adjoint_identity}
%  appears only as a norm equality in many settings, and for $i=j$ Theorem \ref{th:T_adjoint} can be
%  seen as an expression of the symmetry of $T$ (and in fact relates closely to the symmetry of the Grunsky coefficients).
%  It is particularly surprising to us that no version of Theorem \ref{th:Bergman_comparison_restriction}
%   or of Theorem \ref{th:T_adjoint} for $i \neq j$ is stated by Schiffer in any setting.
%  \end{remark}

 Here are a few words on terminology. The Beurling transform in the plane is defined by $$B_\mathbb{C} f(z)=\frac{-1}{\pi}\mathrm{PV}\iint_{\mathbb{C}}\frac{f(\zeta)}{(z-\zeta)^2}\, dA(\zeta).$$ Schiffer refers to this operator as the Hilbert transform, due to the fact that the operator in question behaves like the actual Hilbert transform $$\mathcal{H}f(x):= \mathrm{PV}\int_{\mathbb{R}}\frac{f(y)}{x-y}\,dy.$$  The term ``Hilbert transform'' is also the one used in O. Lehto's classical book on Teichm\"uller theory \cite{Lehto}. Indeed the integrands of both operators exhibit a similar type of singularity in their respective domains of integration and both fall into the general class of Calder\' on-Zygmund singular integral operators. For such operators, one has quite a complete and satisfactory theory, both in the plane and on differentiable manifolds.

 We shall refer to the restriction of the Beurling transform to anti-holomorphic functions on
 fixed domain as the {\it{Schiffer operator}}.  Here, of course, we express this equivalently as an operator on anti-holomorphic one-forms.
\end{subsection}

\begin{subsection}{Identities for comparison operators}

 \begin{theorem}  \label{th:Bergman_comparison_restriction}
  Let $R$ be a compact surface and let $\Gamma$ be a strip-cutting Jordan curve separating $R$ into two components,
  one of which is $\riem$.  Then $S(\riem)= \mathrm{Res}(\riem)^*$,  where $^\ast$ denotes the adjoint operator.
 \end{theorem}
 \begin{proof}
  Let $\alpha \in A(\riem)$ and $\beta \in A(R)$.  Then, using the reproducing property of $K_R$,
  \begin{align*}
   (S(\riem) \alpha, \beta)_R & = \iint_{R,z} \iint_{\riem,\zeta} K_R(z,\zeta) \wedge_\zeta \alpha(\zeta)
    \wedge_z \overline{\beta(z)} \\
    & = \iint_{\riem_,\zeta} \iint_{R,z} \overline{K_R(\zeta,z)} \wedge_z \overline{\beta(z)} \wedge
     \overline{\alpha(\zeta)} \\
    & = \iint_{\riem,\zeta} \overline{\beta(\zeta)} \wedge \alpha(\zeta) =(\alpha,\mathrm{Res}\,\beta)_{\riem}.
  \end{align*}
  Note that interchange of order of integration is legitimate by Fubini's theorem, due to the analyticity and boundedness of the the Bergman kernel.
 \end{proof}

 Define
 \begin{align*}
  \overline{T}(\riem_j,\riem_k):A(\riem_j) & \rightarrow \overline{A(\riem_k)} \\
  h & \mapsto \overline{T(\riem_j,\riem_k) \overline{h}}.
 \end{align*}
 and similarly for $\overline{S}(\riem_k)$.

  \begin{theorem} \label{th:T_adjoint} Let $R$ be a compact surface.  Let $\Gamma$ be a strip-cutting Jordan curve with measure zero, and assume that the complement of $\Gamma$ consists of two connected components $\riem_1$ and $\riem_2$.  Then
 \[   T(\riem_j,\riem_k)^\ast  = \overline{T}(\riem_k,\riem_j).  \]

 \end{theorem}
 \begin{proof}
 { If $j =k$, the claim follows
 from the non-singular integral representation (\ref{eq:nonsingular_Schiffer}) and
 interchanging the order of integration.}

 The claim essentially follows from the corresponding fact for planar domains, and we need only reduce the problem to this case using coordinates.
 Denote
 \[  \mathcal{L}_\mathbb{C}(z,w) = \frac{1}{\pi} \frac{1}{(z-w)^2}.  \]

We first show that for $G, H\in L^2 (\mathbb{C})$ one has
\begin{equation}\label{adjoint identity in the plane}
\iint_{\mathbb{C}} \big(\iint_{\mathbb{C}} \mathcal{L}_{\mathbb{C}}(z, w) \overline{H(z)} \, dA(z)\big)\,\overline{G(w)}\, dA(w) = \iint_{\mathbb{C}} \big(\iint_{\mathbb{C}} \mathcal{L}_{\mathbb{C}}(z, w) \overline{G(w)} \, dA(w)\big)\,\overline{H(z)}\, dA(z)
\end{equation}
where the inside integral is understood as a principle value integral in both cases.

Now, for $f \in L^2(\mathbb{C})$, the Beurling transform is given by
 \begin{equation}\label{beurling}
B_\mathbb{C} f(z)= \mathrm{PV}\iint_{\mathbb{C}}{\mathcal{L}}_{\mathbb{C}}(z, \zeta)\,f(\zeta)\, dA(\zeta)= \frac{-1}{\pi}\mathrm{PV}\iint_{\mathbb{C}}\frac{f(\zeta)}{(z-\zeta)^2}\, dA(\zeta),
 \end{equation}

With this notation, and denoting $\overline{H}(w) = \overline{H(w)}$,  (\ref{adjoint identity in the plane}) amounts to
\begin{equation}\label{adjoint identity in the plane for T}
\iint_{\mathbb{C}} B_{\mathbb{C}} \overline{H}(w)\,\overline{G}(w)\, dA(w) = \iint_{\mathbb{C}}
B_{\mathbb{C}} \overline{G}(z)\,\overline{H}(z)\, dA(z).
\end{equation}
If one defines the Fourier transform through
 $$\widehat{f}(\xi,\eta)=\iint_{\mathbb{R}^2} e^{-2\pi i(x\xi+y\eta)}\, f(x+iy)\, dx\, dy,$$ then one has that $\widehat{B_\mathbb{C} f}(\xi,\eta)=\frac{\xi-i\eta}{\xi+i\eta}\,\widehat{f}(\xi,\eta).$

 Using Parseval's formula and the above Fourier multiplier representation of the Beurling transform, one has that
$$\iint_{\mathbb{C}} B_\mathbb{C} \overline{H}(w)\,\overline{G}(w)\, dA(w)= \iint_{\mathbb{C}} \widehat{B_\mathbb{C} \overline{H}}(\xi,\eta)\,\widehat{\overline{G}}(\xi,\eta)\,  d\xi\,d\eta = \iint_{\mathbb{C}} \frac{\xi-i\eta}{\xi+i\eta}\, \widehat{\overline{H}}(\xi,\eta)\,\widehat{\overline{G}}(\xi,\eta)\, d\xi\,d\eta,$$
and
$$\iint_{\mathbb{C}} B_\mathbb{C} \overline{G}(z)\,\overline{H}(z)\, dA(z)= \iint_{\mathbb{C}} \widehat{B_\mathbb{C} \overline{G}}(\xi,\eta)\,\widehat{\overline{H}}(\xi,\eta)\, d\xi\,d\eta= \iint_{\mathbb{C}} \frac{\xi-i\eta}{\xi+i\eta}\, \widehat{\overline{G}}(\xi,\eta)\,\widehat{\overline{H}}(\xi,\eta)\, d\xi\,d\eta.$$
This proves (\ref{adjoint identity in the plane for T}) and hence \eqref{adjoint identity in the plane}.

Now let $B$ be a doubly-connected neighbourhood of $\Gamma$ and $\phi:B \rightarrow U \subseteq \mathbb{C}$ be a doubly-connected chart.   Let $E= B \cap \riem_1$ and $E'=B \cap \riem_2$.  Then $\riem_1 = D \cup E$ and $\riem_2 = D' \cup E'$ for some compact sets $D \subset \riem_1$ and $D' \subseteq \riem_2$ whose shared boundaries with $E$ and $E'$ are strip-cutting Jordan curves.  We may choose these as regular as desired (say, analytic Jordan curves, which in particular have measure zero).
Observe that we then have, for any forms $\alpha \in A(\riem_2)$ and $\beta \in A(\riem_1)$
\begin{equation}\label{first crap}
\begin{split}
\iint_{\riem_1} \iint_{\riem_2}L (\zeta,\eta) \wedge_\zeta \overline{\alpha(\zeta)} \wedge_\eta \overline{\beta(\eta)}\\ = \Big(\iint_{D} \iint_{D'} + \iint_{D} \iint_{E'}  + \iint_{E} \iint_{D'} + \iint_E \iint_{E'}\Big) L (\zeta,\eta) \wedge_\zeta \overline{\alpha(\zeta)} \wedge_\eta \overline{\beta(\eta)}.
\end{split}
\end{equation}
and

\begin{equation}\label{second crap}
\begin{split}
\iint_{\riem_2} \iint_{\riem_1}  L (\zeta,\eta) \wedge_\eta \overline{\beta(\eta)} \wedge_\zeta \overline{\alpha(\zeta)}  \\ = \Big(\iint_{D'} \iint_{D} + \iint_{D'} \iint_{E}  + \iint_{E'} \iint_{D} + \iint_{E'} \iint_{E}\Big)L (\zeta,\eta) \wedge_\eta \overline{\beta(\eta)} \wedge_\zeta \overline{\alpha(\zeta)}.
\end{split}
\end{equation}
We only need to show that one can interchange integrals in each term.
The first three integrals in the right hand side of \eqref{first crap} are equal to their interchanged counterparts in the first three terms of \eqref{second crap}. This follows from Fubini's theorem, using the fact that $L(z,\zeta)$ is non-singular and in fact bounded on all of the six domains of integration involved in those integrals. Therefore it is enough to show that
\begin{equation*}
  \iint_E \iint_{E'} L (\zeta,\eta) \wedge_\zeta \overline{\alpha(\zeta)} \wedge_\eta \overline{\beta(\eta)}  = \iint_{E'}\iint_E L (\zeta,\eta)\wedge_\eta \overline{\beta(\eta)} \wedge_\zeta \overline{\alpha(\zeta)} .
\end{equation*}

To show this, let $\phi$ be a local coordinate with $\eta=\phi(w)$ and $\zeta = \phi(z)$.  We pull back the integral to the plane under $\psi=\phi^{-1}$ so that we reduce the problem to showing that
\begin{equation}  \label{eq:temp3}
 \iint_{\phi(E)} \iint_{\phi(E')} (\psi \times \psi)^* L (\zeta,\eta) \wedge_\zeta \psi^* \overline{\alpha(\zeta)} \wedge_\eta \psi^* \overline{\beta(\eta)}  = \iint_{\phi(E')}\iint_{\phi(E)} (\psi \times \psi)^*L (\zeta,\eta)\wedge_\eta \psi^* \overline{\beta(\eta)} \wedge_\zeta\psi^*  \overline{\alpha(\zeta)} .
\end{equation}
Recall that in local coordinates by equation (\ref{eq:L_expressed_locally})
\[  (\psi \times \psi^*) L(\zeta,\eta)= \left( -\frac{1}{2\pi i} \frac{1}{(\zeta-\eta)^2} + H(\eta)
 \right) d\zeta \, d\eta,\]
where $H(\eta)$ is holomorphic near $\zeta$.  For the holomorphic error term, we can just apply Fubini's theorem, so matters reduce to the demonstration of \eqref{eq:temp3} for the principal term of $\mathcal{L}_\mathbb{C}(\zeta,\eta)$ which contains the singularity.  We may write
$\psi^*\alpha(z) = h(z) dz$ and $\psi^*\beta(w) = g(w) dw$
for some $L^2$ holomorphic functions $g$ on $E$ and $h$ on $E'$.  So the problem is reduced to showing that
\[  \iint_{\phi(E')} \iint_{\phi(E)} \mathcal{L}_{\mathbb{C}}(z,w) \wedge_\zeta \overline{h(z)} \wedge_\eta \overline{g(w)} dA(z) dA(w)  := \iint_{\phi(E')} \iint_{\phi(E)}
\mathcal{L}_{\mathbb{C}}(z,w)\overline{h(z)} \overline{g(w)} dA(w) dA(z).  \]
Letting
\begin{equation*}
G(z)=
\begin{cases}
g(z), \,\,\, z\in E\\
0, \,\,\, z\in \mathbb{C}\setminus E
\end{cases}
\end{equation*}
and
\begin{equation*}
H(z)=
\begin{cases}
h(z), \,\,\, z\in E'\\
0, \,\,\, z\in \mathbb{C}\setminus E'
\end{cases}
\end{equation*}
then $G$ and
$H$ are $L^2$ on $\mathbb{C}$ and the claim now follows directly from  \eqref{adjoint identity in the plane}.
\end{proof}

%  The idea is to control this difference in terms of an integral
% \[ \iint  \iint_{(\riem_2 \backslash \riem_2^s) \times (\riem_2 \backslash \riem_2^t)}  L \overline{h} \overline{g}   \]
% or even if we set
% \[   U_{s,t} =   ((\text{cl} \riem_2) \backslash \riem_2^s) \cup \
%    ((\text{cl} \riem_1) \backslash \riem_1^t)  \]
% \[  \iint_{U_{s,t}} \iint_{
% U_{s,t}}  L \overline{h} \overline{g}. \]
% This could be done for $h$ and $g$ which extend holomorphically
% a little way past $\Gamma$.   Perhaps using the fact that the
% measure of the region of integration goes to zero will do the trick.\\

%  \begin{todo}  Notational problem above.  Maybe use wedge products with differentials of two variables?
%   Perhaps
%   \[  \iint_{\riem_j} \iint_{\riem_i}  L_R(\zeta,\eta) \wedge_\zeta \overline{u(\zeta)} \wedge_\eta \overline{v(\eta)}  \]
%   and we have to make a note that $\zeta$ and $\eta$ forms commute.  Or, one can insist on placing brackets in the integrals initially.
%  \end{todo}
%  \begin{todo}  This also works for $i=j$, doesn't it?
%  \end{todo}
%  \begin{todo}  Indicate which variable is being integrated over?
%  \end{todo}

%  \begin{remark}  Actually I don't know if this appears in Bergman and Schiffer.
%   It seems that Schiffer rarely considered the adjoint of a kernel function.
%  \end{remark}

 We also have the following identity.
 \begin{theorem}  \label{th:two_kernels_adjoint_identity} If $\Gamma$ is a quasicircle then
  \[  T(\riem_1,\riem_1)^\ast T(\riem_1,\riem_1) + T(\riem_1,\riem_2)^\ast T(\riem_1,\riem_2) + \overline{S}(\riem_1)^* \overline{S}(\riem_1) = I.   \]
 \end{theorem}
 \begin{proof}
By Theorem \ref{th:T_adjoint}, and interchange of order of integration (which can be justified as in the proof of Theorem \ref{th:T_adjoint}) we have that
  \begin{align*}
    [T(\riem_1,\riem_2)^* T(\riem_1,\riem_2) \alpha](z)
    & = \iint_{\riem_2,\zeta}  \overline{L_R(z,\zeta)} \wedge_\zeta
    \iint_{\riem_1,w} L_R(\zeta,w) \wedge_w \alpha(w) \\
    & = \iint_{\riem_1,w} \left( \iint_{\riem_2,\zeta}   \overline{L_R(z,\zeta)} \wedge_\zeta
     L_R(\zeta,w) \right)  \wedge_w \alpha(\zeta)
  \end{align*}
  so the integral kernel of $T(\riem_1,\riem_2)^* T(\riem_1,\riem_2)$ is
  \[  \iint_{\riem_2,\zeta}   \overline{L_R(z,\zeta)} \wedge_\zeta
     L_R(\zeta,w).  \]
  Similarly, by equation (\ref{eq:nonsingular_Schiffer}) and Theorem \ref{th:T_adjoint},  the integral
  kernel of $T(\riem_1,\riem_1)^* T(\riem_1,\riem_1)$ is
  \[   \iint_{\riem_1,\zeta}  \left( \overline{L_R(z,\zeta)} - \overline{L_{\riem_1}(z,\zeta)}\right) \wedge
   \left(L_R(\zeta,w) - L_{\riem_1}(\zeta,w)  \right).    \]
  Finally,  by Theorem \ref{th:Bergman_comparison_restriction}, the integral kernel of $\overline{S}(\riem_1)^* \overline{S}(\riem_1)$ is $\overline{K_R(z,w)}$.

    Using this and the reproducing property of $K_\riem$ we need only demonstrate the following identity:
  \begin{align}  \label{eq:the_needful}
    \iint_{\riem_1,\zeta}  &\left( \overline{L_R(z,\zeta)} - \overline{L_{\riem_1}(z,\zeta)}\right) \wedge
   \left( L_R(\zeta,w) - L_{\riem_1}(\zeta,w)  \right) \nonumber \\ & +
     \iint_{\riem_2,\zeta} \overline{L_R(z,\zeta)} \wedge
      L_R(\zeta,w) = \overline{K_{\riem_1}(z,w)} - \overline{K_R(z,w)}.
  \end{align}

  Fix $w \in \riem_1$ and orient $\Gamma^w_\epsilon$
  positively with respect to $\riem_1$.  For fixed $w$, $\partial_w g_{\riem_1}(\zeta,w)$ goes to zero uniformly as $\epsilon \rightarrow 0$. We then have that (applying (\ref{th:Schiffer_vanishing_identity})
  \begin{align*}
   \iint_{\riem_1,\zeta}  &\left( \overline{L_R(z,\zeta)} - \overline{L_{\riem_1}(z,\zeta)}\right) \wedge
   \left( L_R(\zeta,w) - L_{\riem_1}(\zeta,w)  \right)   \\ & = \iint_{\riem_1,\zeta}  \left(  \overline{L_R(z,\zeta)} - \overline{L_{\riem_1}(z,\zeta)}\right) \wedge L_R(\zeta,w)  \\
   & = \lim_{\epsilon \rightarrow 0}  -\frac{1}{\pi i} \int_{\Gamma^w_\epsilon} \left(  \overline{L_R(z,\zeta)} - \overline{L_{\riem_1}(z,\zeta)}\right) \partial_w g(\zeta,w) \\
   & =  - \frac{1}{\pi i} \lim_{\epsilon \rightarrow 0} \int_{\Gamma^{w}_\epsilon}
   \overline{L_R(z,\zeta)}\, \partial_w g(\zeta,w)
    + \frac{1}{\pi i} \lim_{\epsilon \rightarrow 0} \int_{\Gamma^{w}_\epsilon} K_{\riem_1}(z,\zeta)
   \, \partial_w g(\zeta,w) \\
  \end{align*}
  where we have applied equation (\ref{eq:level_curve_identity}) in the last step.

  Applying Stokes' theorem to the first term, we see that
  \[ - \frac{1}{\pi i} \lim_{\epsilon \rightarrow 0} \int_{\Gamma^w_\epsilon}
   \overline{L_R(z,\zeta)}\, \partial_w g(\zeta,w) = - \iint_{\riem_2,\zeta} \overline{L_R(z,\zeta)} \wedge L_R(\zeta,w).   \]
   Here we used the fact that quasicircles have measure zero.  Note that $\Gamma^w_\epsilon$ is negatively oriented with respect to $\riem_2$.
 For the second term, we have
  \begin{align*}
    \frac{1}{\pi i} \lim_{\epsilon \rightarrow 0} \int_{\Gamma^w_\epsilon} K_{\riem_1}(z,\zeta)
   \, \partial_w g(\zeta,w) & = \frac{1}{\pi i} \lim_{\epsilon \rightarrow 0} \int_{\Gamma^w_\epsilon}
   K_{\riem_1}(z,\zeta)  \left( \partial_w g(\zeta,w) - \partial_w g_{\riem_1}(\zeta,w) \right) \\
   & = - \iint_{\riem_1,\zeta} K_{\riem_1}(z,\zeta) \wedge \left( \overline{K_R(\zeta,w)} - \overline{K_{\riem_1}(\zeta,w)}
   \right) \\
   & = - \overline{K_R(z,w)} + \overline{K_{\riem_1}(z,w)}
  \end{align*}
  where in the last term we have used part (5) of Proposition \ref{pr:bunch_o_identities}
  and the reproducing property of Bergman kernel on $\riem_1$.
 \end{proof}
 \begin{remark}
 Theorem \ref{th:two_kernels_adjoint_identity} (in various settings)
 appears only as a norm equality in the literature.
 \end{remark}
\end{subsection}
\end{section}
\begin{section}{Jump formula on quasicircles and related isomorphisms}
\label{se:jump}
\begin{subsection}{The limiting integral in the jump formula}

 In this section, we show that the jump formula holds when
 $\Gamma$ is a quasicircle.  We also prove that in this case the Schiffer operator $T(\riem_1,\riem_2)$
 is an isomorphism, when restricted to a certain subclass of $\overline{A(\riem_1)}$.

To establish a jump formula, we would like to define a Cauchy-type integral for elements $h \in \mathcal{H}(\Gamma)$.
  Since $\Gamma$ is not necessarily rectifiable, instead we replace the integral over $\Gamma$ with
  an integral over approximating curves $\Gamma^{p_1}_\epsilon$ (defined at the beginning of Section \ref{Sec:Schiffer's comparison operators}), and use the harmonic
  extensions $\tilde{h} \in \mathcal{D}_{\text{harm}}(\riem_1)$ of elements of $\mathcal{H}(\Gamma)$.

 It is an arbitrary choice whether to
  approximate the curve from within $\riem_1$ or from within $\riem_2$.  Later, we will show that the
  result is the same in the case that $\Gamma$ is a quasicircle. For now, we choose to approximate from within $\riem_1$.  We thus
  define the operator on elements of $\mathcal{D}_{\text{harm}}(\riem_1)$, which in the end
  will be chosen as the unique extension of a fixed element of $\mathcal{H}(\Gamma).$

  Let $h \in \mathcal{D}_{\text{harm}}(\riem_1)$. Fix $q \in R \backslash \Gamma$ and define
  \begin{equation}  \label{eq:jump_definition}
     J_q(\Gamma)h (z) = - \lim_{\epsilon \searrow 0}
        \frac{1}{\pi i}   \int_{\Gamma^{p_1}_\epsilon} \partial_w g(w;z,q) h(w)
  \end{equation}
  for $z \in R \backslash \Gamma$. Observe that, by definition, the curve $\Gamma^{p_1}_\epsilon$ depends on a fixed point $p_1 \in \riem_1$.  However, we shall show that $J_q(\Gamma)$ is independent
  of $p_1$ in a moment.

First we show that the limit exists.  There are several cases depending on the
  locations of $z$ and $q$.  Assume that $q \in \riem_2$, then for $z \in \riem_2$, we have by Stokes' theorem that
  \begin{equation} \label{eq:J_double_integral_out}
    J_q(\Gamma) h(z) = - \frac{1}{\pi i} \iint_{\riem_1} \partial_w g(w;z,q) \wedge
    \overline{\partial} h(w)
  \end{equation}
  so the limit exists and is independent of $p_1$.
  For $z \in \riem_1$ we proceed as follows; let $\gamma_r$ denote the circle of radius $r$ centered at $z$, positively oriented with respect to $z$, in some fixed chart near $z$.
  By applying Stokes' theorem and the mean value property of harmonic functions we obtain
%  \begin{equation}
% J_q(\Gamma) h (z) = h(z) - \frac{1}{\pi i} \iint_{\riem_1} \partial_w g(w;z,q) \wedge \overline{\partial} h_1(w) \ \ \ z \in \riem_1.
%  \end{equation}
%  To see this, Then it is easily seen (as in the proof of
%  Theorem \ref{th:contour_L_identity}) that
  \begin{align}  \label{eq:J_double_integral_in}
   J_q(\Gamma) h (z) & = - \frac{1}{\pi i} \iint_{\riem_1}  \partial_w g(w;z,q) \wedge \overline{\partial} h(w)
    - \lim_{r \searrow 0} \frac{1}{\pi i} \int_{\gamma_r} \partial_w g(w;z,q) h(w) \nonumber\\
    & = - \frac{1}{\pi i} \iint_{\riem_1}  \partial_w g(w;z,q) \wedge \overline{\partial} h(w)
    + h(z).
  \end{align}
  This shows that the limit exists for $z \in R \backslash \Gamma$ and $q \in \riem_2$
  and is independent of $p$.  In the case that $q \in \riem_1$, we obtain similar expressions, but with the term $h(q)$ added to both integrals.\\

This also shows that
  \begin{lemma} For strip-cutting Jordan curves $\Gamma$, the limit  $(\ref{eq:jump_definition})$ exists and is independent of the
   choice of $p_1$.
  \end{lemma}
  Therefore, in the following we will usually omit mention of the point $p_1$ in defining the level curves, and write simply $\Gamma_\epsilon$.

  \begin{theorem}  \label{th:jump_derivatives}
   Let $\Gamma$ be a strip-cutting Jordan curve in
   $R$.  For all $h \in \mathcal{D}_{\mathrm{harm}}(\riem_1)$ and any $q \in R \backslash \Gamma$,
   \begin{align*}
    \partial J_q(\Gamma) h (z) & = T(\riem_1,\riem_2) \overline{\partial} h(z),  \    & z \in \riem_2\\
    \partial J_q(\Gamma) h(z) & = \partial h(z) + T(\riem_1,\riem_1) \overline{\partial} h(z), \
    & z \in \riem_1 \\
    \overline{\partial} J_q(\Gamma) h(z)& = \overline{S}(\riem_1) \overline{\partial} h(z), \
    & z \in \riem_1 \cup \riem_2. \\
   \end{align*}
  \end{theorem}
  \begin{proof}  Assume first that $q \in \riem_2$. The first claim follows from (\ref{eq:J_double_integral_out}) and the fact that
    the integrand is non-singular.  Similarly for $z \in \riem_2$, the third claim follows
    from (\ref{eq:J_double_integral_out}).

    The second claim follows from Stokes theorem:
    \begin{align} \label{eq:add_q_temp}
     \partial J_q(\Gamma) h(z) & = \partial_z \left( - \frac{1}{\pi i} \lim_{\epsilon \searrow 0}
     \int_{\Gamma_\epsilon} (\partial_w g(w;z,q) - \partial_w g_{\riem}(w,z) )\,h(w) \right) \nonumber \\
     & \ \ \ \   -   \partial_z \lim_{\epsilon \searrow 0} \frac{1}{\pi i} \int_{\Gamma_\epsilon}
     \partial_w g_{\riem} (w,z) \,h(w) \nonumber\\
     & = \partial_z \left( - \frac{1}{\pi i}
     \iint_{\riem_1} (\partial_w g(w;z,q) - \partial_w g_{\riem}(w,z) )\wedge_w \overline{\partial} h(w) \right) \nonumber \\
     & \ \ \ \   -   \partial_z \lim_{\epsilon \searrow 0} \frac{1}{\pi i} \int_{\Gamma_\epsilon}
     \partial_w g_{\riem} (w,z) \,h(w) \nonumber\\
     & = - \frac{1}{\pi i}
     \iint_{\riem_1} (\partial_z \partial_w g(w;z,q) - \partial_z \partial_w g_{\riem}(w,z) )
      \wedge_w \overline{\partial} h(w)  + \partial h(z)
    \end{align}
    where we have used Theorem \ref{th:Greens_reproducing}. Also observe that the fact that the integrand
    of the first term is non-singular and holomorphic in $z$ for each $w\in \riem_1$, and that $$\iint_{\riem_1, w}|( \partial_w g(w;z,q) - \partial_w g_{\riem}(w,z) ) \wedge_w \overline{\partial}_{{w}} h(w)|$$ is locally bounded in $z$, yield that derivation under the integral sign in the first term is legitimate.

    Similarly removing the singularity using $\partial_w g_{\riem}$, and then applying
    Theorem \ref{th:Greens_reproducing} and Stokes' theorem yield that
    \begin{align*}
      \overline{\partial} J(\Gamma) h(z) & = - \overline{\partial}_{{z}} \frac{1}{\pi i} \lim_{\epsilon \searrow 0}
     \int_{\Gamma_\epsilon} ( \partial_w g(w;z,q) -  \partial_w g_{\riem}(w,z) )\,h(w)  + \overline{\partial} h(z) \\
     & = - \frac{1}{\pi i} \iint_{\riem_1}(\overline{\partial}_{{z}} \partial_w g(w;z,q) - \overline{\partial}_{{z}} \partial_w g_{\riem}(w,z) ) \wedge_w \overline{\partial}_{{w}} h(w)  + \overline{\partial} h(z).
    \end{align*}
    The third claim now follows by observing that the second term in the integral is just $- \overline{\partial} h$ because the integrand is just the complex conjugate of the Bergman kernel.

    Now assume that $q \in \riem_1$.  We show the second claim in the theorem.  We argue as in equation (\ref{eq:add_q_temp}), except that we must also add
    a term $\partial_w g_{\riem_1}(w;q) h(w)$.  We obtain instead
    \[ \partial J(\Gamma) h =  \partial_z J(\Gamma) h = - \frac{1}{\pi i}
     \iint_{\riem_1} (\partial_z \partial_w g(w;z,q) - \partial_z \partial_w g_{\riem}(w;z) )
      \wedge_w \overline{\partial}_{{w}} h(w)  + \partial_z \left( h(z) + h(q) \right)  \]
     and the claim follows from $\partial_z h(q) =0$.  The remaining claims follow similarly.
  \end{proof}

  Below, let $\overline{A(R)}^\perp$ denote the orthogonal complement in $\overline{A_{\text{harm}}(\riem_1)}$of the restrictions of $\overline{A(R)}$ to $\riem_1$
  \begin{corollary}   \label{co:boundedness_and_holomorphicity}
    Let $\Gamma$ be a strip-cutting Jordan curve and assume that $q \in R \backslash \Gamma$.
   \begin{enumerate}
    \item[$(1)$] $J_q(\Gamma)$ is a bounded operator from $\mathcal{D}_{\mathrm{harm}}(\riem_1)$ to $\mathcal{D}_{\mathrm{harm}}(\riem_1 \cup \riem_2)$.
    \item[$(2)$] If $\overline{\partial} h \in \overline{A(R)}^\perp$ then $J_q(\Gamma) h \in \mathcal{D}(\riem_1 \cup \riem_2)$.
   \end{enumerate}
  \end{corollary}
  \begin{proof} The first claim follows immediately from Theorems \ref{th:T_boundedness} and \ref{th:jump_derivatives}.
   The second claim follows from Theorem \ref{th:jump_derivatives} together with the fact that for fixed $z$
   $\overline{\partial}_{{z}} \partial_w g \in {A(R)}$.
  \end{proof}
 \end{subsection}
 \begin{subsection}{Density theorems}
  In this section we show that certain subsets of the Dirichlet space are dense.

  Our first density result follows from a theorem of N. Askaripour and T. Barron \cite{AskBar}.  In brief, their result says that $L^2$ holomorphic one-forms (in fact, differentials) on a region in a Riemann surface can be approximated by holomorphic one-forms on a larger domain. More specifically they proved the following:
  \begin{proposition}\label{prop:AskBar}
   Let $\mathcal{B}_1$, $\mathcal{B}_2$ be nonempty open subsets of a Riemann surface $\Sigma$ such that $\mathcal{B}_1 \subset \mathcal{B}_2$. For a positive integer $k$ and $j = 1, 2$
denote by $\mathcal{A}^{(k)}_j$
the Hilbert space of holomorphic square-integrable $k$-differentials
on $\mathcal{B}_j $. If now $\iota^* : \mathcal{A}^{(k)}_2\to \mathcal{A}^{(k)}_1$, is defined by $\iota^* s= s|_{\mathcal{B}_1},$ then $\iota^*(\mathcal{A}^{(k)}_2)$ is dense in $\mathcal{A}^{(k)}_1.$
  \end{proposition}

  We need a result of this form for the Dirichlet space.  Equivalently, one must show that exact $L^2$ one-forms can be approximated by exact $L^2$ one-forms on the larger region.

  \begin{theorem} \label{th:slightly_bigger_density}
   $(1)$ Let $R$ be a Riemann surface and $\riem_1$ be bounded by a strip-cutting Jordan curve in $R$ which
   separates $R$.  Let $\riem_1'$ contain $\riem_1$ and be such $\riem_1' \backslash \mathrm{cl}\, \riem_1$ is a collar neighbourhod of $\Gamma$.   Let $\mathrm{Res}:A(\riem_1') \rightarrow A(\riem_1)$ denote restriction.  Then $\mathrm{Res}\, A_e(\riem_1')$ is dense in $A_e(\riem_1)$.\\

   $(2)$ Let $R$ be a compact Riemann surface and $\Gamma$ be a strip-cutting Jordan curve.  Let $U$ be a
    doubly-connected neighbourhood of $\Gamma$.
    Let $A_i = U \cap \riem_i$ for $i=1,2$, and let $\mathrm{Res}_i:\mathcal{D}(U) \rightarrow \mathcal{D}(A_i)$ denote restriction for $i=1,2$.  Then $\mathrm{Res}_i \mathcal{D}(U)$ is dense in $\mathcal{D}(A_i)$ for
   $i=1,2$.
  \end{theorem}
  \begin{proof}  Assume that $R$ is a hyperbolic Riemann surface.  Applying Proposition \ref{prop:AskBar} with $k=1$, $\mathcal{B}_1 = \riem_1$, $\mathcal{B}_2 = \riem_1'$, and $\riem=R$,
   we see that $\mathrm{Res}\, A(\riem_1')$ is dense in $A(\riem_1)$.  If on the other hand $R$ is not hyperbolic, let $\riem$ be obtained by removing a conformal disk from $R$ whose closure is disjoint from
   the closure of $\riem_1'$.  Let $k$, $\mathcal{B}_1$ and $\mathcal{B}_2$ be as before.  The conditions
   of Proposition \ref{prop:AskBar} are then easily verified, and we have that $\mathrm{Res}\, A(\riem_1')$ is dense in $A(\riem_1)$.

We need to show that the claim holds for exact forms.  Choose a set of closed contours
   $\gamma_1,\ldots,\gamma_m$ in $\riem_1$ generating the fundamental group of $\riem_1$.
   For $\alpha \in A(\riem_1)$ let
   \[ Q(\alpha) = \left(  \int_{\gamma_1} \alpha,\ldots, \int_{\gamma_m} \alpha  \right).  \]
  Each component of  $Q \circ \text{Res}$ is a bounded linear functional.  To see this, for
  fixed $i$ let $\phi:B \rightarrow U$ be a local holomorphic chart in a neighbourhood $B$ of $\gamma_i$.
   Now ${\Gamma_i} = \phi \circ \gamma_i$ is a compact subset of $U$, and has finite length $l$ say.  Let
   $M_{\Gamma_i}$ be obtained by applying
  Lemma \ref{th:uniform_controlled_by_L^2} to $A(U)$ and ${\Gamma_i}$.  We then have that
   \begin{align*}
      \left| \int_{\gamma_i} \alpha \right|&  = \left| \int_{\phi \circ \gamma_i} (\phi^{-1})^* \alpha
   \right|  \leq M_{{\Gamma_i}} l \| (\phi^{-1})^* \alpha \|_{A(U)} =  M_{\Gamma_i} l \|  \alpha \|_{A(B)}  \\
  & \leq M_{\Gamma_i} l   \| \alpha \|_{A(\riem_1')}.
   \end{align*}
   Let $\mathcal{P}:A(\riem_1') \rightarrow \text{Ker}\, Q$ denote the
   orthogonal projection onto
   the kernel of $Q \circ \text{Res}$.
   By using the Riesz representation theorem and the Gram-Schmidt process, one can show that there is a $C$ such that
   \[   \|  \mathcal{P} \alpha - \alpha \|_{A(\riem_1')} \leq C \| Q \circ \text{Res}(\alpha) \|_{\mathbb{C}^m}.  \]
   Again using Lemma \ref{th:uniform_controlled_by_L^2} to control the uniform norm on $\gamma_1 \cup \cdots \cup \gamma_m$
   by the norm on $A(\riem_1)$, there is also an $M$ such that
   \[  \| Q (\alpha) \|_{\mathbb{C}^m} \leq M \| \alpha \|_{A(\riem_1)},  \]
   for all $\alpha \in A(\riem_1)$.

   Now let $\beta \in A_e(\riem_1)$ and let $\varepsilon >0$.   Choose $\alpha \in A(\riem_1')$ such
  that $\| \beta - \text{Res}\, \alpha \|_{A(\riem_1)} <\varepsilon$.  Using the fact that $Q(\beta)=0$ together with the two preceding bounds, we have
   \begin{align*}
    \| \text{Res}\, \alpha - \text{Res}\, \mathcal{P} \alpha \|_{A(\riem_1)} & \leq
     \| \alpha - \mathcal{P} \alpha \|_{A(\riem_1')} \leq C \| Q(\text{Res}\, \alpha) \|_{\mathbb{C}^m} \\
      & =  C \| Q(\text{Res}\, \alpha - \beta) \|_{\mathbb{C}^m} \leq C \cdot M \cdot \| \text{Res}\, \alpha - \beta \|_{A(\riem_1)} \\ & \leq C M \varepsilon.
   \end{align*}
   Thus
   \begin{align*}
    \| \beta - \text{Res}\, \mathcal{P} \alpha \|_{A(\riem_1)} & \leq \| \beta - \text{Res}\,
     \alpha \|_{A(\riem_1)}  + \| \text{Res}\, \alpha - \text{Res}\, \mathcal{P} \alpha \|_{A(\riem_1)} \\
      & \leq (1 + CM) \varepsilon.
   \end{align*}
   Since $\text{Res}\, \mathcal{P} \alpha \in A_e(\riem_1),$ this completes the proof of the first claim.

  To prove the second claim, fix $i=1$ or $2$ and set $\mathcal{B}_1 = A_i$, $\mathcal{B}_2= U$, in Proposition \ref{prop:AskBar} and let $\riem = R$ be as above (with disks removed if necessary as above). Choose a single
   curve $\gamma$ in $A_i$ which is homotopic to $\Gamma$, and define $Q(\alpha) = \int_\gamma \alpha$.
    The proof now proceeds identically, but with a single curve.
  \end{proof}
  \begin{remark} We will only use the second part of Theorem \ref{th:slightly_bigger_density}.  However, part (a) is no more difficult than
  part (b), and
  will be useful in future applications.
  \end{remark}
  \begin{remark}  Note that the result of Askaripour and Barron \cite{AskBar} will not in general
  apply to exact one-forms without
   restrictions on the relation between the homology of $\riem_1$ and $\riem_1'$.
   It is clear that weaker conditions could be imposed, so that we have not used the full power of
   their result.
  \end{remark}
%  It immediately follows from Theorem \ref{th:slightly_bigger_density} part (1) that the set of restrictions of elements of $\mathcal{D}(\riem_1')$ to $\riem_1$
%  is dense in $\mathcal{D}(\riem_1)$.

  We will also need a density result of another kind.
  Let $\Gamma$ be a strip-cutting Jordan curve in a compact Riemann surface $R$, which
  separates $R$ into two components $\riem_1$ and $\riem_2$.  Let $A$ be a collar neighbourhood of $\Gamma$ in $\riem_1$.
  By Theorem \ref{th:transmission_equivalences} the
  boundary values of $\mathcal{D}_{\text{harm}}(A)$ exist conformally non-tangentially in $\riem_1$ and are themselves CNT boundary values of an element of $\mathcal{D}_{\text{harm}}(\riem_1)$.  We then define
  \begin{align*}
   \mathfrak{G}: \mathcal{D}_{\text{harm}}(A) & \rightarrow \mathcal{D}_{\text{harm}}(\riem_1)  \\
   h & \mapsto \tilde{h}
  \end{align*}
  where $\tilde{h}$ is the unique element of $\mathcal{D}_{\text{harm}}(\riem_1)$ with CNT boundary
  values equal to those of $h$.   We have the following result:

  \begin{theorem}[\cite{Schippers_Staubach_general_transmission}] \label{th:iota_bounded} Let $\Gamma$ be a strip-cutting Jordan curve in a compact Riemann surface $R$.  Assume
  that $\Gamma$ separates $R$ into two components, one of which is $\riem$.  Let $A$ be a
   collar neighbourhood of $\Gamma$  in  $\riem$.  Then the associated map $\mathfrak{G}:\mathcal{D}_{\mathrm{harm}}(A) \rightarrow \mathcal{D}_{\mathrm{harm}}(\riem)$ is bounded.
  \end{theorem}

  \begin{theorem} \label{th:Dirichlet_annulus_dense} Let $\Gamma$, $R$, $A$ and $\riem$ be as above.  The
   image of $\mathcal{D}(A)$ under $\mathfrak{G}$ is dense in $\mathcal{D}_{\mathrm{harm}}(\riem_1)$.
  \end{theorem}
  \begin{proof}
   First, we prove this in the case that $A = \mathbb{A}$ is an annulus with outer boundary $\mathbb{S}^1$ and $\riem_1 = \disk$, and $\mathfrak{G}$ is
   \[  \mathfrak{G}(\mathbb{A},\mathbb{D}):\mathcal{D}_{\text{harm}}(\mathbb{A}) \rightarrow \mathcal{D}_{\text{harm}}(\disk).  \]
   Now the set of Laurent polynomials $\mathbb{C}[z,z^{-1}]$ are contained in  $\mathbb{D}_{\text{harm}}(\mathbb{A})$, and
   \[  \mathfrak{G}(\mathbb{A},\disk) z^n = z^n \ \ \text{and} \ \
   \mathfrak{G}(\mathbb{A},\disk) z^{-n}  = \bar{z}^n.  \]
   Since the set $\mathbb{C}[z,\bar{z}]$ of polynomials in $z$, $\bar{z}$ is dense in $\mathcal{D}_{\text{harm}}(\disk)$, this proves the claim.

   Next, let $F:\mathbb{A} \rightarrow A$ be a conformal map.  Define the
   composition map
   \begin{align*}
     \mathcal{C}_F :\mathcal{D}_{\text{harm}}(A) & \rightarrow \mathcal{D}_{\text{harm}}(\mathbb{A}) \\
     h & \mapsto h \circ F,
   \end{align*}
   which is bounded by conformal invariance of the Dirichlet norm, and furthermore is a bijection with bounded inverse $\mathcal{C}_{F^{-1}}$.  Similarly the restriction of $\mathcal{C}_F$ to $\mathcal{D}(A)$ is a bounded bijection onto $\mathcal{D}(\mathbb{A})$.
   Thus, the image of $\mathcal{D}(A)$ under
   $\mathfrak{G}(\mathbb{A},\mathbb{D}) \mathcal{C}_F$ is dense in $\mathcal{D}_{\text{harm}}(\disk)$.

   Now denote the restriction map from $\mathcal{D}_{\text{harm}}(\disk)$
   to $\mathcal{D}_{\text{harm}}(\mathbb{A})$ by $\text{Res}(\disk,\mathbb{A})$
   and similarly for $\text{Res}(\riem_1,A)$.  Define the linear map
   \[ \rho = \mathfrak{G}(A,\riem_1) \, \mathcal{C}_{F^{-1}} \, \text{Res}(\disk,\mathbb{A}) : \mathcal{D}_{\text{harm}}(\disk)
   \rightarrow \mathcal{D}_{\text{harm}}(\riem_1). \]
   This is obviously bounded, with bounded inverse
   \[ \rho^{-1} = \mathfrak{G}(\mathbb{A},\mathbb{D})\, \mathcal{C}_F\, \text{Res}(\riem_1,A) \]
   by uniqueness of Dirichlet bounded harmonic extensions of elements of $\mathcal{H}(\mathbb{S}^1)$ and $\mathcal{H}(\Gamma)$ to $\mathcal{D}_{\text{harm}}(\disk)$ and $\mathcal{D}_{\text{harm}}(\riem_1)$
   respectively.

   Now by definition of $\mathfrak{G}(\mathbb{A},\disk)$, for any $h \in \mathcal{D}_{\text{harm}}(A)$, the $\mathrm{CNT}$ boundary values of
   \[  \mathcal{C}_{F^{-1}} \, \text{Res}(\disk,\mathbb{A})\,\mathfrak{G}(\mathbb{A},\disk)\,
   \mathcal{C}_F h \] equal those of $h$.  Thus we obtain the following factorization of $\mathfrak{G}(A,\riem_1)$:
   \[   \rho \, \mathfrak{G}(\mathbb{A},\mathbb{D}) \mathcal{C}_F =
      \mathfrak{G}(A,\riem_1) \mathcal{C}_{F^{-1}}\, \text{Res}(\disk,\mathbb{A})
       \,\mathfrak{G}(\mathbb{A},\disk) \, \mathcal{C}_F = \mathfrak{G}(A,\riem_1).
      \]
   Since the image of $\mathcal{D}(A)$ under $\mathfrak{G}(\mathbb{A},\mathbb{D}) \mathcal{C}_F$ is dense in $\mathcal{D}_{\text{harm}}(\disk)$, and $\rho$
   is a bounded bijection with bounded inverse, this completes the proof.

  \end{proof}

 \end{subsection}
 \begin{subsection}{Limiting integrals from two sides}
  In this section, we show that for quasicircles, the limiting integral defining $J_q(\Gamma)$
  can be taken from either side of $\Gamma$, with the same result.

 We will need to write the limiting integral in terms of holomorphic extensions to collar neighbourhoods.
  The integral in the definition of
  $J_q(\Gamma)$ is easier to work with when restricting to $\mathcal{D}(A)$.  To make use of this simplification, we must first show that the limiting
  integrals of $\mathfrak{G} h$ and $h$ are equal.

  For $h \in \mathcal{D}(A)$, letting $\Gamma_\epsilon$ be level curves of Green's
  function of $\riem_1$ with respect to some fixed point $p \in \riem_1$, denote
  \[  J_q(\Gamma)'h(z) = - \frac{1}{\pi i} \lim_{\epsilon \searrow 0} \int_{\Gamma_\epsilon}
    \partial_w g(w;z,q) h(w) \]
  for $q$ fixed in $\riem_2$.
  We use the notation $J_q(\Gamma)'$ to distinguish it from the operator
  $J_q(\Gamma)$, which applies only to elements of $\mathcal{D}_{\text{harm}}(\riem_1)$.
  For $\epsilon$ in some interval $(0,R)$ the curve $\Gamma_\epsilon$
 lies entirely in $A$, so this makes sense.  Because the integrand is holomorphic,
 the integral is independent of $\epsilon$ for  $\epsilon \in (0,R)$.

 We have the following.
 \begin{theorem}  \label{th:integral_in_DofA_equal} Let $\Gamma$ be a quasicircle
 and $A$ be a collar neighbourhood of $\Gamma$ in $\riem_1$.  Fix $q \in R \backslash \Gamma$.
 Then for all $h \in \mathcal{D}(A)$ and $z \in R \backslash \Gamma$
  \begin{equation} \label{eq:equality_of_integrals_annulus}
    J_q(\Gamma)'h(z)  = J_q(\Gamma) \mathfrak{G} h(z).
  \end{equation}

 \end{theorem}
 \begin{proof}
  First we exhibit a dense subset of functions for which this is true.  Let $f:\{ z\,:\, r<|z|<1 \} \rightarrow A$ be the
  inverse of a canonical
  collar chart.   Letting $\mathbb{C}(z)$ denote the Laurent polynomials in $z$, define
  \[  \mathcal{D}_{\text{poly}}(A) = \{  p \circ f^{-1} \,:\, p \in \mathbb{C}(z) \} \subset \mathcal{D}(A).    \]
  Now $\mathcal{D}_{\text{poly}}(A)$ is dense in $\mathcal{D}(A)$, because the polynomials are dense in $\mathcal{D}(\mathbb{A})$
  where $\mathbb{A} = \{z : r<|z|<1 \}$, and $\mathcal{C}_f h = h \circ f$ is an isometry.
  If one fixes any point $x \in A$, and denotes $\mathcal{D}_{\text{poly}}(A)_x = \{ u \in
  \mathcal{D}_{\text{poly}}(A): u(x)=0 \}$
   then $\mathcal{D}_{\text{poly}}(A)_x$ is dense in $\mathcal{D}(A)_x$.
   Since $J_q(\Gamma)$ and $J_q'(\Gamma)$ clearly agree on constants,
   we may assume that elements of
   $\mathcal{D}(A)$ are so normalized if need be.

  We claim that (\ref{eq:equality_of_integrals_annulus}) holds for all $h \in \mathcal{D}_{\text{poly}}(A)$.  To see this,
  observe that for any $\epsilon$, $f^{-1}$ takes the curve $\Gamma_\epsilon$ to the curve
  $|z|=e^{-c\epsilon}$ for some $c>0$ which is independent of $\epsilon$.  So for $z \in \Gamma_\epsilon$, $\overline{f^{-1}(z)}^n = e^{-2nc\epsilon} (f^{-1}(z))^{-n}$ for any $n \in \mathbb{Z}$.   Thus
  \begin{align*}
       \frac{1}{\pi i} \lim_{\epsilon \searrow 0} \int_{\Gamma_\epsilon}\partial_w g(w;z,q)
       \overline{f^{-1}(w)}^n & =  \frac{1}{\pi i} \lim_{\epsilon \searrow 0} e^{-2n c \epsilon} \int_{\Gamma_\epsilon}\partial_w g(w;z,q)  (f^{-1}(w))^{-n} \\
       & = \frac{1}{\pi i} \lim_{\epsilon \searrow 0} \int_{\Gamma_\epsilon}\partial_w g(w;z,q)  (f^{-1}(w))^{-n}
  \end{align*}
  where we use the fact that the integrand on the right side of the first line is holomorphic in $w$
  and thus the integral is independent of $\epsilon$.  This proves the claim.

 By Lemma \ref{th:uniform_controlled_by_L^2}, for compact $K$ there is an $M_K$ such that
 for all $g \in \mathcal{D}_{\text{harm}}(A)_x$
  \begin{equation}  \label{eq:uniform_Dirichlet_control}
  \sup_{z\in K}|g(z)| \leq M_K \| g \|_{\mathcal{D}_{\text{harm}}(A)}.
  \end{equation}

   The theorem now follows from the boundedness of $\mathfrak{G}$ and of $J_q(\Gamma)$.  To see this let $h \in \mathcal{D}(A)$.  Given $u \in \mathcal{D}_{\text{poly}}(A)$ we have that
  \[  \| J_q(\Gamma) \mathfrak{G} u - J_q(\Gamma) \mathfrak{G} h \|_{\mathcal{D}_{\text{harm}} (\riem_1 \cup\riem_2)}
     \leq \| J_q(\Gamma) \| \| \mathfrak{G} \| \| u - h \|_{\mathcal{D}(A)}.  \]
     Fixing any $z \in \riem_1 \cup \riem_2$,  let $K$ be a compact set containing $z$.  By \eqref{eq:uniform_Dirichlet_control}, the norm estimate above, and the normalization of $J_q(\Gamma)$
      we have that for any
     $\varepsilon >0$, there is a $\delta_1 >0$ so that
     \begin{equation}  \label{eq:temp_second_term_epsilon}
        \sup_{z\in K}|J_q(\Gamma) (\mathfrak{G} u -  \mathfrak{G} h)(z)| < \varepsilon/2
     \end{equation}
   whenever $\| u - h \|_{\mathcal{D}(A)} < \delta_1$.

   Because the integrands of $J_q(\Gamma)' u$ and $J_q(\Gamma)' h$ are holomorphic,
   the limiting integral is equal to the integral over $\Gamma'$ for some fixed level curve $\Gamma'= \Gamma_\epsilon \subset A$.   Recall (see above) that we can assume that
   $h$ and $u$ are in $\mathcal{D}(A)_x$ for some fixed $x \in A$.
   Again by (\ref{eq:uniform_Dirichlet_control}), there is a $M_{\Gamma'}$ so that $\sup_{z\in \Gamma'}| (u - h)(z) | \leq M_{\Gamma'} \| u - h \|_{\mathcal{D}(A)}$.
  Using the fact that $\partial_w g(w;z,q)$ is bounded on the compact set $K$ for $w \in \Gamma'$, there is a $\delta_2$ such that $ \| u - h \|_{\mathcal{D}(A)} < \delta_2$ implies that
  \begin{equation} \label{eq:temp_first_term_epsilon}
   \sup_{z\in K}|J_q(\Gamma)' (u-h) (z)| < \varepsilon/2.
  \end{equation}

  Using the fact that $J_q(\Gamma)' u = J_q(\Gamma) \mathfrak{G} u$, we have that
  \[  \sup_{z\in K} |J_q(\Gamma)' h (z)- J_q(\Gamma) \mathfrak{G} h (z) | \leq \sup_{z\in K}| J_q(\Gamma)'(h-u)(z) | + \sup_{z\in K}|J_q(\Gamma) (\mathfrak{G} u - \mathfrak{G} h)(z) |.    \]
  Combining this with (\ref{eq:temp_first_term_epsilon}) and (\ref{eq:temp_second_term_epsilon}),
  by the density of $\mathcal{D}(A)_{\text{poly}}$ we see that
  \[ \sup_{z\in K}   |J_q(\Gamma)' h (z)- J_q(\Gamma) \mathfrak{G} h (z) | < \varepsilon   \]
  for all $\varepsilon >0$, which proves the theorem.

%   {\color{blue} *******

%   The above argument seems to work.  There is another similar argument,  which appears to fail, because I don't know whether
%   $J(\Gamma)'$ is bounded on $\mathcal{D}(A)$.  I include it for now
%   even though it doesn't seem to be necessary or correct.}
%   Let $w \in W$ be such that $\| w - h \|_{\mathcal{D}(A)} < \epsilon$.  Then
%   \begin{align*}
%    \| J(\Gamma)' h - J(\Gamma) \mathfrak{G} h \|_{\mathcal{D}(\riem_1 \cup \riem_2)_{\text{harm}}}  &
%     = \| J(\Gamma)' h - J(\Gamma)' w + J(\Gamma) \mathfrak{G} w - J(\Gamma) \mathfrak{G} h \|_{\mathcal{D}(\riem_1 \cup \riem_2)_{\text{harm}}}  \\
%     & \leq \| J(\Gamma)' \| \| h -  w \|_{\mathcal{D}(A)} + \| J(\Gamma)  \| \| \mathfrak{G} \|
%     \| w - h \|_{\mathcal{D}(A)}.
%   \end{align*}
%   Since $W$ is dense in $\mathcal{D}(A)$, this proves the claim.
 \end{proof}
 Since every collar neighbourhood of $\Gamma$ contains a canonical collar neighbourhood $\Gamma$, we have

\begin{corollary}  Proposition $\ref{th:integral_in_DofA_equal}$ holds for any collar neighbourhood
  $A$ of $\Gamma$ in $\riem_1$.
 \end{corollary}

 Finally we observe that the same proof, replacing $\partial_w g(w;z,q)$ with any holomorphic
  one-form $\alpha$ on a neighbourhood of $\Gamma$ shows the following.
  \begin{theorem}  \label{th:general_annulus_holo_integral_same}
   Let $\Gamma$ be a quasicircle.  Let $\alpha$ be a holomorphic one-form
    on an open neighbourhood of $\Gamma$ in $R$.  Let $A$ be a collar neighbourhood of $\Gamma$ in $\riem_1$.  Let $\Gamma_\epsilon$ be level
    curves of Green's function in $\riem_1$.
    Then for all $h \in \mathcal{D}(A)$,
    \[   \lim_{\epsilon \searrow 0} \int_{\Gamma_\epsilon} \alpha(w) h(w) =
     \lim_{\epsilon \searrow 0} \int_{\Gamma_\epsilon} \alpha(w) \mathfrak{G} h(w).   \]
  \end{theorem}

   We now show that for quasicircles, one can define the jump operator $J(\Gamma)$ using either
   limiting integrals from within $\riem_1$ or from within $\riem_2$ with the same result.  We use the following temporary notation.  For $q \in R \backslash \Gamma$ let $J_q(\Gamma,\riem_i):\mathcal{D}_{\text{harm}}(\riem_1)
   \rightarrow \mathcal{D}_{\text{harm}}(\riem_1 \cup \riem_2)$ be defined
   by
   \[    J_q(\Gamma,\riem_i)h (z) = - \lim_{\epsilon \searrow 0}
        \frac{1}{\pi i}   \int_{\Gamma^{p_i}_\epsilon} \partial_w g(w;z,q) h(w).   \]
   For definiteness, we assume that all curves $\Gamma^{p_i}_\epsilon$ are oriented positively
   with respect to {\color{blue}$\riem_i$}.  Aside from this change of sign, all previous theorems apply
   equally to $J_q(\Gamma,\riem_1)$ and $J_q(\Gamma,\riem_2)$.
  \begin{theorem} \label{th:jump_reflection_invariant} Let $\Gamma$ be a quasicircle.  Then for all $q \in R \backslash \Gamma$
    \[ J_q(\Gamma,\riem_1) = J_q(\Gamma,\riem_2) \mathfrak{O}(\riem_1,\riem_2). \]
   \end{theorem}
   \begin{proof}
    Let $U$ be a doubly-connected neighbourhood of $\Gamma$, bounded by
    $\Gamma_i \subset \riem_i$.  Let $A_i = U \cap \riem_i$.
    Then $A_i$ are collar neighbourhoods of $\Gamma$ in $\riem_i$. Let  $\mathfrak{G}_i: \mathcal{D}(A_i) \rightarrow \mathcal{D}_{\text{harm}}(\riem_i)$ be induced by these collar neighbourhoods for
    $i=1,2$.

    For any $h \in \mathcal{D}(U)$, let $\text{Res}_i\, h = \left. h \right|_{A_i}$.
    It follows immediately from the definition of $\mathfrak{G}_i$ that
    \begin{equation} \label{eq:reflection_true}
      \mathfrak{G}_2 \text{Res}_2\, h =
     \mathfrak{O}(\riem_1,\riem_2) \mathfrak{G}_1 \text{Res}_1\, h.
    \end{equation}

    Therefore
    \begin{align*}
     \lim_{\epsilon \searrow 0} \int_{\Gamma_\epsilon^{p_1}} \partial_w g(w;z,q) \mathfrak{G}_1 h(w)
     & = \lim_{\epsilon \searrow 0} \int_{\Gamma_\epsilon^{p_1}} \partial_w g(w;z,q) h(w) \\
     & = \lim_{\epsilon \searrow 0} \int_{\Gamma_\epsilon^{p_2}} \partial_w g(w;z,q) h(w) \\
     & = \lim_{\epsilon \searrow 0} \int_{\Gamma_\epsilon^{p_2}} \partial_w g(w;z,q) \mathfrak{G}_2 h(w),
    \end{align*}
    where we have used holomorphicity of the integrand in the second equality, and
    Proposition \ref{th:integral_in_DofA_equal} to obtain the first and the third  equalities.  Thus
    for all $h \in \mathfrak{G}_i \text{Res}_i\, \mathcal{D}(U)$,
    \[  J(\Gamma,\riem_1) h = J(\Gamma,\riem_2) \mathfrak{O} (\riem_1,\riem_2) h.  \]

    Now by Theorem
    \ref{th:slightly_bigger_density} $\text{Res}_i\, \mathcal{D}(U)$ is dense in $\mathcal{D}(A_i)$ for $i=1,2$ , and therefore by Theorem \ref{th:Dirichlet_annulus_dense} part (2)
    $\mathfrak{G}_i \text{Res}_i\, \mathcal{D}(U)$ is dense in $\mathcal{D}_{\text{harm}}(\riem_i)$.
    Since $\text{Res}_i$, $\mathfrak{G}_i$ and $J(\Gamma,\riem_i)$ are all bounded, this completes the proof.
   \end{proof}

   Thus one may think of $J(\Gamma)$ as an operator on $\mathcal{H}(\Gamma)$.

   In the rest of the paper, we return to the convention that $J_q(\Gamma)$ is an operator
   on $\mathcal{D}_{\text{harm}}(\riem_1)$.  However, Theorem \ref{th:jump_reflection_invariant} plays an important role in the proof that $T(\riem_1,\riem_2)$ is surjective.

   Also, by using Theorem \ref{th:general_annulus_holo_integral_same} and proceeding exactly
   as in the proof of Theorem \ref{th:jump_reflection_invariant} we obtain
   \begin{theorem} \label{th:holomorphic_sides_equal}  Let $\Gamma$ be a quasicircle.  Let $\alpha$ be a holomorphic one-form
    in an open neighbourhood of $\Gamma$.  For any $h \in \mathcal{D}_{\mathrm{harm}}(\riem_1)$
    \[ \lim_{\epsilon \searrow 0} \int_{\Gamma^{p_1}_\epsilon} h(w) \alpha(w)
     =   \lim_{\epsilon \searrow 0} \int_{\Gamma^{p_2}_\epsilon} [\mathfrak{O}(\riem_1,\riem_2) h](w) \alpha(w) \]
   \end{theorem}
% We also have the following crucial theorem.
% \begin{corollary}  \label{co:Cauchy_kernel_either_side}
%  Let $R$ be a compact Riemann surface.  Let $\Gamma$ be a quasicircle dividing $R$ into
%  components $\riem_1$ and $\riem_2$.  Choose $p_i \in \riem_i$, $i=1,2$ and let $\Gamma_{i,\epsilon}$ be the level curves of Green's function $g_{\riem_i,p_i}$ of $\riem_i$.  If $g(w;z,q)$ is Green's function of $R$.  For $h \in \mathcal{H}(\Gamma)$ let $h_i$ denote the extensions of $h$ in $\mathcal{D}(\riem_i)$.  Fix $q \in R \backslash \Gamma$.   Then
%  \[   \lim_{\epsilon \searrow 0} \frac{1}{2 \pi} \int_{\Gamma_{1,\epsilon}}
%   \partial_w g(w;z,q)  h_1(w) =  \lim_{\epsilon \searrow 0} \frac{1}{2 \pi} \int_{\Gamma_{2,\epsilon}}
%   \partial_w g(w;z,q)  h_2(w).   \]
% \end{corollary}
% \begin{proof}
%  This follows immediately from Theorem \ref{th:integral_either_side_general} with $\alpha = \partial_w g(w;z,q)$.
% \end{proof}

\end{subsection}
\begin{subsection}{A transmission formula}
 In this section we prove an explicit formula for the transmission operator $\mathfrak{O}$ on the image of  the jump operator.
\begin{definition}\label{def: W}
 We denote by $W_k$ the linear subspace of $\mathcal{D}_{\text{harm}}(\riem_i)$ given by
  \[ W_k = \left\{ h \in \mathcal{D}_{\text{harm}}(\riem_k) \,:\, \lim_{\epsilon \searrow 0} \int_{\Gamma^{{p_k}}_\epsilon} h(w) \alpha(w) =0  \right\}\]
  for all $\alpha \in A(R)$ and for $k=1,2$. The elements of $W_k$ are the admissible functions for the jump problem.
\end{definition}  Let
 \[  J(\Gamma)_{\riem_k} h = \left. J(\Gamma) h \right|_{\riem_k}  \]
 for $k=1,2$.  We have the following result:
 \begin{theorem}  \label{th:reflection_special_case}
  Let $R$ be a compact surface and $\Gamma$ be a quasicircle separating
  $R$ into components $\riem_1$ and $\riem_2$.  Let $q \in R \backslash \Gamma$. If $h \in W_1$
  then
  \[  - \mathfrak{O}(\riem_2,\riem_1) J_q(\Gamma)_{\riem_2}   h  =   h - J_q(\Gamma)_{\riem_1}
      h.  \]
 \end{theorem}

  To prove this theorem we need a lemma.
 \begin{lemma}\label{Lem:Prametrix lemma}  Let $\Gamma$ be a quasicircle and let $A$ be a collar neighbourhood of $\Gamma$ in $\riem_1$.
  Fix a smooth curve $\Gamma'$ in $A$ homotopic to $\Gamma$, and assume that $h \in \mathcal{D}(A)$ satisfies
  \begin{equation} \label{eq:DofA_jump_condition}
   \int_{\Gamma'} h \alpha =0
  \end{equation}
  for all $\alpha \in A(R)$.  Then $\mathfrak{G}  h \in W_1$ and
  \[  - \mathfrak{O}(\riem_2,\riem_1) J_q(\Gamma)_{\riem_2} \mathfrak{G} h  = \mathfrak{G} h - J_q(\Gamma)_{\riem_1}
    \mathfrak{G}   h.  \]
 \end{lemma}
  \begin{proof}  The fact that $\mathfrak{G} h \in W_1$ follows immediately from
  Theorem \ref{th:general_annulus_holo_integral_same}.
   By Royden \cite[Theorem 4]{Royden} and the explicit formula on the following page, there
   are holomorphic functions $H_1$ on $\riem_1$ and $H_2$ on $\text{cl}\, \riem_2 \cup A$ such that
   $H_1 - H_2 = h$ on $A$.  Furthermore, these functions are given by
    the restrictions of $J_q(\Gamma)' h$ to $\riem_1$ and $\riem_2$.  Thus, by Proposition \ref{th:integral_in_DofA_equal}, we have that
    \begin{equation} \label{eq:Hi_are_what_we_need}
     H_k = \left. J_q(\Gamma) \right|_{\riem_k} \mathfrak{G} h
    \end{equation}
    for $k=1,2$ (where $H_2$ has a holomorphic extension to
    $\text{cl}\riem_2 \cup A$).

    Since $H_1$, $H_2$ and $h$ are all in $\mathcal{D}(A)$, they have conformally
    non-tangential boundary values in $\mathcal{H}(\Gamma)$ with respect to $\riem_1$.  Since $H_1 - H_2 = h$
    on $A$, then the boundary values also satisfy this equation.  Thus
    \[ H_1 - \mathfrak{O}(\riem_2,\riem_1) H_2 = \mathfrak{G} h \]
    by definition of $\mathfrak{G}$ and $\mathfrak{O}(\riem_2,\riem_1)$.  Finally
    equation (\ref{eq:Hi_are_what_we_need}) completes the proof.
  \end{proof}
   We continue with the proof of Theorem \ref{th:reflection_special_case}.
  \begin{proof}
   Let $E$ be the linear subspace of $\mathcal{D}(A)$ consisting of those elements of $\mathcal{D}(A)$
   for which (\ref{eq:DofA_jump_condition}) is satisfied.  It suffices to show that $\mathfrak{G} E$ is
   dense in $W_1$.

   Fix a basis $\alpha_1,\ldots \alpha_g$ for $A(R)$.   Let $\mathcal{P}:\mathcal{D}(A) \rightarrow E$
   denote the orthogonal projection in $\mathcal{D}(A)$.

   For $u \in \mathcal{D}(A)$ define
   \[ Q(u) = \left( \int_{\Gamma'}u \alpha_1, \ldots, \int_{\Gamma'} u \alpha_g \right).   \]
   By Lemma \ref{th:uniform_controlled_by_L^2}  and the fact that $Q(u + c) = Q(u)$ for
   any constant $c$, it follows that
   each component of $Q$ is a bounded linear functional on  $\mathcal{D}(A)$.
   Once again, a simple argument based on Riesz representation theorem and the Gram-Schmidt process yields that there is a $C$ such that
   \begin{equation} \label{eq:plausiblemaybe}
     \| \mathcal{P} u - u  \|_{\mathcal{D}(A)} \leq C \| Q(u) \|_{\mathbb{C}^g}.
   \end{equation}
%  This amounts to show the following:\\ Let $L_1,\ldots,L_g$ be bounded linear functionals on $\mathcal{D}(A)$ and let $E = \cap_{i=1}^{g}\text{Ker}\, L_i$.  Denote by $\mathcal{P}_{E^\perp}$ the
%    orthogonal projection onto $E^\perp$ (which is finite dimensional).
%  Then one has
%    \begin{equation}\label{abstract estimate}
%    \| \mathcal{P}_{E^\perp} u \|_{ E^\perp} \leq C
%    \| (L_1 u,\ldots,L_g u) \|_{\mathbb{C}^g} \end{equation}

%    for all $u \in \mathcal{D}(A).$ Indeed, by the Riesz representation theorem, there is a unique $h_i\in \mathcal{D}(A)$ such that $L_i u=(h_i, u)$. This yields that $\text{Ker}\, L_i= (\text{span}\,h_i)^\perp.$ Moreover $E^\perp= \mathrm{cl}(\sum_{i=1}^g \mathrm{span}\,h_i).$ Now, let $v_i$, $i=1,\dots,\, \mathrm{dim} E^\perp$, be an orthonormal base for the subspace $E^\perp$ which can be produced using the Gram-Schmidt process. Then $P_{E^\perp}u= \sum_{i=1}^{\mathrm{dim} E^\perp} (v_i, u) v_i.$ However, $v_i=\sum_{j=1}^{g} c_{ij} h_j$ and therefore $P_{E^\perp}u= \sum_{i=1}^{\mathrm{dim} E^\perp}\sum_{j=1}^g c_{ij} (L_j u) v_i$, from which \eqref{abstract estimate} follows by using the orthonormality of $v_i$'s and the Cauchy-Schwarz inequality.

   For $H \in \mathcal{D}(\riem_1)_{\text{harm}}$ define now
   \[ Q_1(H) = \lim_{\epsilon \searrow 0}
    \left( \int_{\Gamma^{p_1}_\epsilon}H \alpha_1, \ldots, \int_{\Gamma^{p_1}_\epsilon} H \alpha_g \right). \]
    We have that there is a $C'$ such that
    \[  \| Q_1(H) \|_{\mathbb{C}^g} \leq C' \| H \|_{\mathcal{D}_{\text{harm}}(\riem_1)}.   \]
    This follows by applying Stokes' theorem to each component:
    \[   \lim_{\epsilon \searrow 0} \int_{\Gamma^{p_1}_\epsilon} H \alpha_k =
      \iint_{\riem_1}   \overline{\partial} H \wedge \alpha_k   \]
    which is proportional to $(\overline{\partial} H, \overline{\alpha_k})_{A(\riem_1)}$.
    Observe also that $Q_1( \mathfrak{G} u) = Q(u)$ for all $u \in \mathcal{D}(A)$ by Proposition
    \ref{th:integral_in_DofA_equal}.

      Let $h \in W_1 \subseteq \mathcal{D}_{\text{harm}}(\riem_1)$ be arbitrary.
   By density of $\mathfrak{G} \mathcal{D}(A)$, there is a $u \in \mathcal{D}(A)$ such
   that
   \[  \| \mathfrak{G} u - h \|_{\mathcal{D}_{\text{harm}}(\riem_1)}< \varepsilon.  \]
   We then have
   \begin{align*}
      \| \mathfrak{G} \mathcal{P} u - h \|_{\mathcal{D}_{\text{harm}}(\riem_1)}
      & \leq \| \mathfrak{G} \mathcal{P} u - \mathfrak{G} u \|_{\mathcal{D}_{\text{harm}}(\riem_1)} +
      \| \mathfrak{G} u - h \|_{\mathcal{D}_{\text{harm}}(\riem_1)} \\
    & \leq \| \mathfrak{G} \|  \| \mathcal{P} u - u \|_{\mathcal{D}(A)}
     + \| \mathfrak{G} u - h \|_{\mathcal{D}_{\text{harm}}(\riem_1)}.
   \end{align*}

   Now
   \[  \| Q(u) \| = \| Q_1(\mathfrak{G} u) \| = \| Q_1(\mathfrak{G} u - h) \|
    \leq C' \| \mathfrak{G} u - h \| < C' \varepsilon \]
   so by (\ref{eq:plausiblemaybe})
   \[  \| \mathcal{P} u - u \|_{\mathcal{D}(A)} \leq C C' \varepsilon.  \]
   Thus
   \[  \| \mathfrak{G} \mathcal{P} u - h \|_{\mathcal{D}_{\text{harm}}(\riem_1)}
     \leq (C C'\| \mathfrak{G} \|   +1 ) \varepsilon.  \]
  \end{proof}

 We also define a transmission operator for exact one-forms as follows:
\begin{definition}\label{defn: transmission of exact}
 For an
 exact one-form $\alpha \in A_e(\riem_2)_{\text{harm}}$ let $h_2$ be a harmonic function on $\riem_2$ such that $dh_2= \alpha$.  Let $h_1$ be the unique element of
 $\mathcal{D}(\riem_1)_{\text{harm}}$ with boundary values agreeing with $h_2$.
 Then we define
 \begin{align*}
  \mathfrak{O}_e(\riem_2,\riem_1):A_e(\riem_2)_{\text{harm}} & \rightarrow
   A_e(\riem_1)_{\text{harm}} \\
   \alpha & \mapsto d h_1.
 \end{align*}
 The transmission from $A_e(\riem_1)_{\text{harm}}$ to $A_e(\riem_2)_{\text{harm}}$
 is defined similarly.
 \end{definition}

 To prove the transmission formula for $\mathfrak{O}_e$, we require the following
 elementary lemma.
 \begin{lemma} \label{le:dbar_solvability}
  Let $\riem$ be a Riemann surface of finite genus $g$ bordered by
  a curve homeomorphic to a circle.  Let $\overline{\alpha} \in \overline{A(\riem)}$.
  There is an $h \in \mathcal{D}_{\mathrm{harm}}(\riem)$ such that $\overline{\partial} h =
  \overline{\alpha}$.  If $\tilde{h} \in \mathcal{D}_{\mathrm{harm}}(\riem)$ is
  any other such function, then $\tilde{h} - h \in \mathcal{D}(\riem)$.
 \end{lemma}
 \begin{proof}  Let $R$ be the double of $\riem$; so $A(R)$ has dimension $2g$ where $g$
 is the genus of $\riem$.
  Let $a_1,\ldots,a_{2g}$ be a collection of smooth curves which generate the fundamental group of $\riem$.   Let
  \[ c_k = \int_{a_k}  \overline{\alpha} \]
  for $k=1,\ldots,2g$.  Since $A(R)$ has dimension $2g$, there is a $\beta \in A(R)$
  such that
  \[ \int_{a_k} \beta = -c_k \]
  for $k=1,\ldots,2g$.  Thus $\overline{\alpha} + \beta$ is exact in $\riem$ and hence is equal to $dh$ for some $h \in \mathcal{D}_{\text{harm}}(\riem)$.  But clearly $\overline{\partial} h =\overline{\alpha}$.

  If $\tilde{h}$ is any other such function then $\overline{\partial} (\tilde{h} - h) =0$,
  which completes the proof.
 \end{proof}

 Recall that $\overline{A(R)}^\perp$ denotes the set of elements in $A_{\text{harm}}(\riem)$
 which are orthogonal to the restrictions of elements of $\overline{A(R)}$ with respect to $(\cdot,\cdot)_{A_{\text{harm}}}(\riem)$.

 \begin{definition}\label{def: V and Vprime}
  Given $R$ and $\riem_i$ as above, let
 \[  V_k = \overline{A(\riem_k)} \cap \overline{A(R)}^\perp,  \]
 and
 \[   V_k' =  \{ \overline{\alpha} + \beta \in A_{\text{harm}}(\riem_k)_e \,:\, \overline{\alpha} \in
     V_k\},   \]
     for $k=1,2$.
 \end{definition}
 \begin{theorem}  \label{th:transmission_formula_with_T}
  Let $R$ be a compact Riemann surface and let $\Gamma$ be a quasicircle separating $R$ into components $\riem_1$ and $\riem_2$. If $\overline{\alpha} \in V_1$ then
  \[ - \mathfrak{O}_e (\riem_2,\riem_1) T(\riem_1,\riem_2) \overline{\alpha}
     = \overline{\alpha} - T(\riem_1,\riem_1) \overline{\alpha}.  \]
 \end{theorem}
 \begin{proof}
  Let $\overline{\alpha} \in V_1$, then  by
  Lemma \ref{le:dbar_solvability} there is an $h \in \mathcal{D}_{\text{harm}}(\riem_1)$ such
  that $\overline{\partial} h = \overline{\alpha}$.

  Since $\overline{\partial} h = \overline{\alpha} \in \overline{A(R)}^\perp$,
  $\overline{S}(\riem_1) \overline{\partial} h = 0$, so by Theorem \ref{th:jump_derivatives}
  $\overline{\partial} J(\Gamma)h =0$.

  Applying $d$ to both sides of (\ref{th:reflection_special_case}) and using this fact yields
  \begin{equation*}
   - \mathfrak{O}_e(\riem_2,\riem_1) {\partial} J(\Gamma)_{\riem_2} h
    = dh - \partial J(\Gamma)_{\riem_1} h.
  \end{equation*}
  The Theorem now follows from the remaining relations in Theorem \ref{th:jump_derivatives}.

 \end{proof}
 For $k=1,2$ denote by
 \begin{align*}
  P(\riem_k): A_{\text{harm}}(\riem_k) & \rightarrow A(\riem_k) \\
  \overline{P}(\riem_k):A_{\text{harm}}(\riem_k) & \rightarrow \overline{A(\riem_k)}
 \end{align*}
 the orthogonal projections onto the holomorphic and anti-holomorphic parts of a given harmonic one-form.
 \begin{corollary}  \label{co:Tonetwo_injectivity}
  Let $R$ be a compact Riemann surface and $\Gamma$ be a quasicircle separating $R$ into components $\riem_1$ and $\riem_2$.  Then $-\overline{P}(\riem_1) \mathfrak{O}(\riem_2,\riem_1)$ is a left inverse of $\left. T(\riem_1,\riem_2) \right|_{V_1}$. In particular, the restriction of $T(\riem_1,\riem_2)$ to $V_1$ is injective.
 \end{corollary}
 \begin{proof}
  This follows immediately from Theorem \ref{th:transmission_formula_with_T} and the fact that for $\overline{\alpha} \in V_1$, $T(\riem_1,\riem_1) \overline{\alpha}$ and $T(\riem_1,\riem_2) \overline{\alpha}$ are holomorphic.
 \end{proof}

 As another consequence of Theorem \ref{th:transmission_formula_with_T} we are able to prove an inequality analogous
 to the strengthened Grunsky inequality for quasicircles \cite{Pommerenkebook}.

 \begin{theorem}  \label{th:Grunsky_souped_up} Let $R$ be a compact Riemann surface and $\Gamma$ be a quasicircle separating
  $R$ into disjoint components $\riem_1$ and $\riem_2$.  Then $\| \left. T(\riem_1,\riem_1) \right|_{V_1} \| <1$.
 \end{theorem}
 \begin{proof}
  Since $d:\mathcal{D}_{\text{harm}}(\riem_k) \rightarrow A_{\text{harm}}(\riem_k)$ is norm-preserving (with respect to the Dirichlet semi-norm), it follows from Theorem \ref{th:transmission_bounded} that there is
 a $c \in (0,1)$ which is independent of $\overline{\alpha}$ such that
 \begin{equation}\label{estim: boundedness of action of O}
 \| \mathfrak{O}_e (\riem_2,\riem_1) \,  T(\riem_1,\riem_2)\overline{\alpha} \|^2 \leq
   \frac{1+c}{1-c} \, \| T(\riem_1,\riem_2) \overline{\alpha} \|^2.
   \end{equation}
   We will insert the identity
  \begin{equation} \label{eq:isomorphism_proof_temphalf}
   -\mathfrak{O}_e (\riem_2,\riem_1) T(\riem_1,\riem_2) \overline{\alpha} = \overline{\alpha}   - T(\riem_1,\riem_1) \overline{\alpha}
  \end{equation}
  of Theorem \ref{th:transmission_formula_with_T} into (\ref{estim: boundedness of action of O}).

In the following computation, we need two observations.  First, if a function $H$ is
  holomorphic on a domain $\Omega$, then $\| H \|^2_{\mathcal{D}_{\text{harm}}}(\Omega) = 2\| \text{Re} (H) \|^2_{\mathcal{D}(\Omega)}$.  Second, if $H_2$ is a primitive of $T(\riem_1,\riem_2) \overline{\alpha}$ and if we let $H_1 = \mathfrak{O}(\riem_2,\riem_1) H_2$ (so that $H_1$ is a primitive of $\overline{\alpha} - T(\riem_1,\riem_1) \overline{\alpha}$ by definition), then we observe that $\mathfrak{O}(\riem_2,\riem_1)
  \text{Re}(H_2) = \text{Re}(H_1)$, and therefore the boundedness of transmission estimate applies to $\text{Re}(H_i)$.

  Since $\alpha - T(\riem_1,\riem_1) \overline{\alpha}$ has the same real part as the right hand side of (\ref{eq:isomorphism_proof_temphalf}), combining with (\ref{estim: boundedness of action of O}) (applied to the real part of the primitives) we obtain
  \begin{align} \label{eq:isomorphism_proof_temp2}
    \frac{1+c}{1-c} \, \| T(\riem_1,\riem_2) \overline{\alpha} \|^2
     & = \frac{1 + c}{1-c} \, 2 \| \text{Re}( H_2) \|^2_{\mathcal{D}_{\text{harm}}(\riem_2)}
     \nonumber \\
     & \geq  2 \| \text{Re} (H_1 )\|^2_{\mathcal{D}_{\text{harm}}(\riem_1)}  \nonumber  \\
     & = 2 \| d \,\text{Re} (H_1) \|^2_{A_{\text{harm}}}(\riem_1)
     = 2 \| \text{Re} (d H_1) \|^2_{A_{\text{harm}(\riem_1)}}  \nonumber \\
     & = 2  \| \text{Re} \left( \overline{\alpha} - T(\riem_1,\riem_1) \overline{\alpha}\right)  \|^2
     = \nonumber
     2 \| \text{Re} \left( \alpha - T(\riem_1,\riem_1) \overline{\alpha}\right)  \|^2 \\
     & = \| \overline{\alpha} \|^2 - 2 \text{Re}\,(T(\riem_1,\riem_1) \overline{\alpha},
      \alpha) + \| T(\riem_1,\riem_1) \overline{\alpha}\|^2.
  \end{align}
  where we have used the fact that $\alpha - T(\riem_1,\riem_1) \overline{\alpha}$ is
  holomorphic.
 By Theorem \ref{th:two_kernels_adjoint_identity} we have that
  \begin{align*}
  \| \overline{\alpha}\|^2  = (\overline{\alpha},\overline{\alpha})
  & = \left( \overline{\alpha},   T(\riem_1,\riem_1)^* T(1,1) \overline{\alpha} + T(\riem_1,\riem_2)^* T(\riem_1,\riem_2) \overline{\alpha}
   \right) \\
  & = \left( \overline{\alpha}, T(\riem_1,\riem_1)^* T(\riem_1,\riem_1) \overline{\alpha} \right)
    + \left( T(\riem_1,\riem_2)^* T(\riem_1,\riem_2) \overline{\alpha}
   \right) \\
   & = \| T(\riem_1,\riem_1) \overline{\alpha} \|^2 + \| T(\riem_1,\riem_2) \overline{\alpha} \|^2.
 \end{align*}

%   \[  \| T(\riem_1,\riem_2) \overline{\alpha} \|^2 = \| \overline{\alpha} \|^2 -
%   \|  T(\riem_1,\riem_1) \overline{\alpha} \|^2.  \]
 Combining this with (\ref{eq:isomorphism_proof_temp2})  yields

  \[  - \frac{1-c}{1+c} \text{Re}(\alpha, T(\riem_1,\riem_1) \overline{\alpha})
     \leq \frac{c}{1+c} \| \overline{\alpha}\|^2    - \frac{1}{1+c}
      \| T(\riem_1,\riem_1) \overline{\alpha} \|^2.  \]
Hence
  \begin{align*}
   -\text{Re} (\alpha, T(\riem_1,\riem_1) \overline{\alpha}) & \leq \frac{c}{1+c}
   \| \overline{\alpha}\|^2 - \frac{2c}{1+c} \text{Re} (\alpha, T(\riem_1,\riem_1) \overline{\alpha}) - \frac{1}{1+c} \| T(\riem_1,\riem_1) \overline{\alpha}\|^2 \\
  & = c  \| \overline{\alpha}\|^2 - \frac{1}{1+c} \| T(\riem_1,\riem_1) \overline{\alpha} + c \alpha \|^2.
  \end{align*}
 Applying this to $e^{i\theta} \alpha$, we see that the same inequality holds with the
  left hand side replaced by $-e^{-2i\theta} \text{Re}(\alpha, T(\riem_1,\riem_1) \overline{\alpha})$
  for any $\theta$.
  So
  \[ |\text{Re} (\alpha, T(\riem_1,\riem_1)\overline{\alpha})| \leq c \| \overline{\alpha}\|^2. \]
  Together with the fact that $T(\riem_1,\riem_1)^*=\overline{T(\riem_1,\riem_1)}$
  this proves the theorem.

 \end{proof}

 \begin{remark}  This gives another proof that $T(\riem_1,\riem_2)$ is injective.  Let $\nu=\|T(\riem_1,\riem_1)\| <1$.  Observe that if $\overline{\alpha} \in \overline{A(\riem_1)}$ is in $V_1$, then since the kernel
of the operator $S(\Sigma_1)$ is holomorphic we have that $\overline{\alpha}
\in \text{Ker}\, \overline{S}(\riem_1)$.   Thus by Theorem \ref{th:two_kernels_adjoint_identity}
\begin{align*}
  \| T(\riem_1,\riem_2) \overline{\alpha} \|_{A(\riem_2)}^2 & = \| \overline{\alpha}\|^2_{\overline{A(\riem_1)}}
     -  \|  T(\riem_1,\riem_1) \overline{\alpha}   \|^2_{\overline{A(\riem_1)}}  \\
     & \geq (1-\nu^2) \| \overline{\alpha} \|^2_{\overline{A(\riem_1)}}.
\end{align*}
Since $1-\nu^2 >0$ this completes the proof.
\end{remark}

\end{subsection}
\begin{subsection}{Isomorphism theorem for the Schiffer operator}
 In this section,  we prove the isomorphism theorem for the Schiffer operators.
 We require two facts.

\begin{theorem} \label{th:isomorphism_theorem} Let $\Gamma$ be a quasicircle.  Then the restriction of $T(\riem_1,\riem_2)$
 to $V_1$ is an isomorphism onto $A(\riem_2)_e$.
\end{theorem}
\begin{proof}
   Injectivity of $T(\riem_1,\riem_2)$ is Corollary \ref{co:Tonetwo_injectivity}.

   We show that $T(\riem_1,\riem_2)(V_1) \subseteq A(\riem_2)_e$. If we take $\overline{\alpha} \in V_1$, then  since
  \[ \iint_{\riem_1,w} \partial_{\bar{z}} \partial_w g(w,w_0;z,q) \wedge \overline{\alpha(w)}
      =  0, \]
   for any fixed $q \in \riem_2$ we have (without loss of generality, because $T(\riem_1,\riem_2)$ is
   independent of $q$)
   \begin{align*}
    T(\riem_1,\riem_2) \overline{\alpha}(z)
   & = - \frac{1}{\pi i} \iint_{\riem_1,z} \partial_z \partial_w g(w,w_0;z,q) \wedge \overline{\alpha(w)} \\
   & = - d_z \frac{1}{\pi i} \iint_{\riem_1,z} \partial_w g(w,w_0;z,q) \wedge \overline{\alpha(w)} \in A(\riem_2)_e,
   \end{align*}
   and therefore $T(\riem_1,\riem_2)(V_1) \subseteq A(\riem_2)_e$.

  To show that $T(\riem_1,\riem_2) (V_1)$ contains $A(\riem_2)_e$, let $\beta \in A(\riem_2)_e$, and let $h$ be the unique element of $\mathcal{D}(\riem_2)_q$ such that
  $\partial h = \beta$.  By Theorem \ref{th:transmission_bounded} there is an $H \in \mathcal{D}_{\text{harm}}(\riem_1)$
  such that $h$ and $H$ have the same boundary values on $\Gamma$.  Now $dH = \delta_1 + \overline{\delta_2}$ for $\delta_1, \delta_2 \in A(\riem_1)$ (specifically, $\delta_1 = \partial H$
  and $\overline{\delta_2} = \overline{\partial} H$). Now by Theorem \ref{th:jump_reflection_invariant} we have
  \begin{align*}   %\label{eq:surjectivity_proof_one}
  \beta(z)  & = - { \partial_z} \lim_{\epsilon \searrow 0} \frac{1}{\pi i} \int_{\Gamma^{p_2}_\epsilon} \partial_wg(z,q;w) h(w) \\
  & = -{ \partial_z} \lim_{\epsilon \searrow 0} \frac{1}{\pi i} \int_{\Gamma^{p_1}_\epsilon} \partial_wg(z,q;w) H(w)  \\  %\label{eq:surjectivity_proof_two}
  & = -{ \partial_z}\frac{1}{\pi i} \iint_{\riem_1} \partial_w g(z,q;w) \wedge \overline{\delta_2(w)} \\
  & = - \frac{1}{\pi i}\iint_{\riem_1} \partial_z \partial_w(z,q;w) \wedge \overline{\delta_2(w)} \nonumber
  \end{align*}
  which proves that $A(\riem_2)_e \subseteq \text{Im}(T_R(\riem_1,\riem_2))$.
  Now we need to show
  that $\overline{\partial} H \in V_1$.

Since $h$ is holomorphic by assumption, we have that $\partial h = dh$, hence
  \begin{align*}
    - \frac{1}{\pi i} \iint_{\riem_1}  \partial_z
      \partial_w g(w,w_0;z,q) \wedge \overline{\partial} H(w) & = \partial h(z)  = dh(z)
       \\ &  = - \frac{1}{\pi i} \iint_{\riem_1}  \partial_z
      \partial_w g(w,w_0;z,q) \wedge \overline{\partial} H(w) \\ & \ \  \ \ -
      \frac{1}{\pi i}
       \iint_{\riem_1}  \overline{\partial}_{{z}}
      \partial_w g(w,w_0;z,q) \wedge \overline{\partial} H(w).
  \end{align*}
    Thus
    \begin{equation} \label{eq:integral_against_bergman_zero}
     - \frac{1}{\pi i} \iint_{\riem_1}  \overline{\partial}_{{z}}
      \partial_w g(w,w_0;z,q) \wedge \overline{\partial} H(w) =0
    \end{equation}
  for all $z \in \riem_2$.  If we now let $\overline{\alpha} \in \overline{A(R)}$, then we have

  \begin{align*}
   (\overline{\partial} H, \overline{\alpha})_{\riem_1} & = - \frac{1}{2} \iint_{\riem_1}
    \overline{\partial} H(w) \wedge \alpha(w) \\
    & = - \frac{1}{2} \iint_{\riem_1,w} \overline{\partial} H(w) \wedge_w
    \iint_{R,z} K_R(w;z) \wedge_z \alpha(z) \\
    & = \frac{1}{2} \iint_{R,z} \alpha(z) \wedge_z \iint_{\riem_1,w} \overline{K_R(z;w)}
    \wedge_w \overline{\partial} H(w)
  \end{align*}
  which is zero by (\ref{eq:integral_against_bergman_zero}).  Thus $\overline{\partial} H \in V_1$
  as claimed.
\end{proof}

\begin{remark}  Although we have only proven that $T(\riem_1,\riem_2)$ is injective for quasicircles,
 we conjecture that this is true in greater generality, as in Napalkov and Yulmakhumetov \cite{Nap_Yulm} in the planar case.
 It would also be of interest to give a proof of surjectivity using their approach.  One would use the adjoint identity
 of Theorem \ref{th:T_adjoint} in place of the symmetry of the $L$ kernel, which is used implicitly in their proof.  One would also need to take into account the topological obstacles as we did above.
\end{remark}

 \begin{proposition}  \label{th:exactness_theorem}
  Let $R$ be a compact Riemann surface and let $\Gamma$ be a quasicircle separating $R$
  into  components $\riem_1$ and $\riem_2$. For any $h \in \mathcal{D}_{\mathrm{harm}}(\riem)$
  such that $\overline{\partial} h \in V_1$
  \[ \partial h + T(\riem_1,\riem_1) \overline{\partial} h \in A(\riem_1)_e. \]
 \end{proposition}
 \begin{proof}
 By Corollary \ref{co:boundedness_and_holomorphicity} we need only show that $\partial h + T(\riem_1,\riem_1) \overline{\partial} h$
 is exact.  As usual let $\Gamma_\epsilon$ be level curves of $g_{\riem_1}$ for fixed $z$.
 Since $L_R$ and hence $T(\riem_1,\riem_1)$ is independent of $q$, we can assume
 that $q \in \riem_2$.
  By Stokes' theorem
  \begin{align*}
     - \frac{1}{\pi i}  &\lim_{\epsilon \rightarrow 0} \int_{\Gamma_{\epsilon}}
    \left( \partial_w g(w;z,q) - \partial_w g_{\riem_1}(w,z)  \right) h(w)  \\
    & = - \frac{1}{\pi i} \iint_{\riem_1} \left( \partial_w g(w;z,q) - \partial_w g_{\riem_1}(w,z) \right) \wedge
    \overline{\partial}  h(w) =: \omega(z).
  \end{align*}
    The  integral on the left hand side exists by \eqref{eq:jump_definition}
    and Theorem \ref{th:Greens_reproducing}.  Thus the right hand side is a well-defined function $\omega(z)$
    on $\riem_1$.

   Thus
   \begin{align*}
    T(\riem_1,\riem_1) \overline{\partial} h (z) & = \partial  \omega(z)  \\
    & = d  \omega(z) - \overline{\partial} \omega(z) \\
    & =  d\omega(z) + \frac{1}{\pi i} \iint_{\riem_1} \partial_{\bar{z}}
    \partial_w g(w;z,q) \wedge \overline{\partial} h(w) - \frac{1}{\pi i} \iint_{\riem_1} \partial_{\bar{z}}
    \partial_w g_{\riem_1}(w,z) \wedge \overline{\partial} h(w) \\
    & = d\omega(z) + 0 + \overline{\partial} h
   \end{align*}
where the middle term vanishes because $\overline{\partial} h \in V_1$, and we have observed that the last term is just
the conjugate of the Bergman kernel applied to $\overline{\partial} h$.
Thus $- \overline{\partial} h + T(\riem_1,\riem_1) \overline{\partial} h$ is exact.  Since $dh = \partial h
+ \overline{\partial} h$ is exact, the claim follows.

 \end{proof}

% {\color{blue}
% \begin{lemma}  \label{le:collar_holomorphic_density} Let $\Gamma$ be a quasicircle separating a compact Riemann surface into two
%  connected components $\riem_1$ and $\riem_2$.  Let $U$ be an open neighbourhood of $\Gamma$
%  which is conformally equivalent to an annulus.  Then $\mathcal{D}(U)$ is dense in $\mathcal{H}(\Gamma)$.
% \end{lemma}
% \begin{proof}
%  Let $f:U \rightarrow A$ be a conformal map onto an annulus.  Then $\Gamma'$ is a quasicircle in the plane.  We can assume that $A$ does not contain $0$ or $\infty$ in its closure. Let $\Omega_+$ be the
%   bounded component of the complement of $\Gamma'$ in $\sphere$ and $\Omega_-$ be the unbounded
%   component of the complement in $\sphere$.  Then the polynomials $\mathbb{C}[z]$ are dense in
%   $\mathcal{D}(\Omega_+)$ and the polynomials $\mathbb{C}[1/z]$ are dense in $\mathbb{D}(\Omega_-)$.
%   Since the decomposition
%   \[ J(\Gamma'):\mathcal{D}(\Omega^+) \rightarrow \mathcal{D}(\Omega^+) \oplus \mathcal{D}(\Omega^-)     \]
%   is an isomorphism, and reflection in $\Gamma'$ is also an isomorphism, this shows that the Laurent polynomials $\mathbb{C}(z)$ are dense in $\mathcal{H}(\Gamma')$.
% \end{proof}}

 The following theorem is in some sense a derivative of the jump decomposition.

 \begin{theorem}  \label{th:derivative_jump_isomorphism}
  Let $R$ be a compact Riemann surface and let $\Gamma$ be a quasicircle separating $R$
  into  components $\riem_1$ and $\riem_2$ and $V_1'$ be given as in \emph{Definition \ref{def: V and Vprime}}.
  \begin{align*}
   \hat{\mathfrak{H}}:V_1' & \rightarrow A(\riem_1)_e \oplus A(\riem_2)_e \\
   dh & \mapsto \left( \partial h + T(\riem_1,\riem_1) \overline{\partial} h,
    -T(\riem_1,\riem_2) \overline{\partial} h \right)
  \end{align*}
  is an isomorphism.
 \end{theorem}
 \begin{proof} First we show surjectivity.
  Let $(\alpha,\beta) \in  A(\riem_1)_e \oplus A(\riem_2)_e$.  By Theorem \ref{th:isomorphism_theorem}, $T(\riem_1,\riem_2)$ is surjective so there is a $\overline{\delta}
   \in V_1$
  such that $T(\riem_1,\riem_2) \overline{\delta} =\beta$.  By Lemma \ref{le:dbar_solvability}
  there is a $\tilde{h} \in \mathcal{D}_{\text{harm}}(\riem_1)$ such that $\overline{\partial} \tilde{h} = -\overline{\delta}$.

  Now set  $\mu = \alpha - \partial \tilde{h} - T(\riem_1,\riem_1) \overline{\partial} \tilde{h}$.
  By construction $\mu$ is holomorphic and it is exact by Proposition \ref{th:exactness_theorem}.
  Let $u \in \mathcal{D}(\riem_1)$ be such that $\partial u = \mu$.  Setting $h = \tilde{h} + u$
  we see that
  \begin{align*}
   \hat{\mathfrak{H}}(dh) & = \left( \partial h + T(\riem_1,\riem_1) \overline{\partial} h,
     -T(\riem_1,\riem_2) \overline{\partial} h \right) \\
     & = \left( \partial \tilde{h} + \mu + T(\riem_1,\riem_1) \overline{\partial} \tilde{h},
      -T(\riem_1,\riem_2) \overline{\partial} \tilde{h} \right)\\
      & = (\alpha,\beta).
  \end{align*}
  Thus $\hat{\mathfrak{H}}$ is surjective.

  Now assume that $\hat{\mathfrak{H}}(dh) = 0$.  The vanishing of the second component yields that $-T(\riem_1,\riem_2) \overline{\partial} h =0$, so by Theorem \ref{th:isomorphism_theorem} we have that $\overline{\partial} h =0$. Thus the vanishing of the first component of $\hat{\mathfrak{H}}(dh)$ yields that $\partial h =0$, hence $dh = 0$.
  \end{proof}
\end{subsection}
\begin{subsection}{The jump isomorphism}
 In this section we establish the existence of a jump decomposition for functions in $\mathcal{H}(\Gamma)$.

%  For a quasicircle
%  $\Gamma$ recall that $W$ denotes the linear subspace of $\mathcal{D}_{\text{harm}}(\riem_1)$ such that
%  \[ \lim_{\epsilon \searrow 0} \int_{\Gamma_\epsilon} h(w) \alpha(w) = 0 \]
%  for all $\alpha \in A(R)$.
 \begin{theorem}  Let $R$ be a compact Riemann surface, and let $\Gamma$ be a
 quasicircle separating $R$ into two connected components $\riem_1$ and $\riem_2$.  Fix $q \in \riem_2$ and let $W_1$ be given as in \emph{Definition \ref{def: W}}.
 Then the map
 \begin{align*}
 {\mathfrak{H}}: \mathcal{D}_{\mathrm{harm}}(\riem_1) & \rightarrow \mathcal{D}(\riem_1) \oplus \mathcal{D}(\riem_2)_q \\
  h &\mapsto \left(\left. J_q(\Gamma) h \right|_{\riem_1},\left. J_q(\Gamma) h \right|_{\riem_2} \right)
 \end{align*}
 is a bounded isomorphism from $W_1$ to $\mathcal{D}(\riem_1) \oplus \mathcal{D}(\riem_2)_q$.
 \end{theorem}
 \begin{proof}
  By Corollary \ref{co:boundedness_and_holomorphicity} we have that the image of the map is
  in $\mathcal{D}(\riem_1) \oplus \mathcal{D}(\riem_2)$.  Now since $g(w_0,w_0;z,q)=0$ by definition
  of $g$, (\ref{eq:g_interchange_both}) yields that $g(w,w_0;q,q)=0$.  Therefore $\partial_w g(w;q,q)=0$ and so
  \[  J_q(\Gamma) h(q) = -\frac{1}{\pi i} \lim_{\epsilon \rightarrow 0} \int_{\Gamma_\epsilon}
     \partial_w g(w,w_0;q,q) h(w) =0.  \]
Thus the image of the map is in $\mathcal{D}(\riem_1) \oplus \mathcal{D}_q(\riem_2)$.

  By Theorem \ref{th:jump_derivatives} $\partial {\mathfrak{H}} h = \hat{\mathfrak{H}}\, dh$, so since $\hat{\mathfrak{H}}$
 is an isomorphism by Theorem \ref{th:derivative_jump_isomorphism}, we only need to
  deal with constants.  If $J_q(\Gamma) h =0$ then $dh=0$ so $h$ is constant on $\riem_1$.
  Since the second component of $\mathfrak{H} h$ vanishes at $q$ we see that $h=0$, so ${\mathfrak{H}}$
  is injective.  Now observe that ${\mathfrak{H}}(h+c) = {\mathfrak{H}}h +(c,0)$ for any constant $c$.
  Thus surjectivity follows from surjectivity of $\hat{\mathfrak{H}}$.
 \end{proof}

 \begin{proposition} \label{th:jump_dependable} Let $R$ be a compact Riemann surface, and let $\Gamma$ be a quasicircle separating
 $R$ into components $\riem_1$ and $\riem_2$.  Assume that $\Gamma$ is positively oriented with
  respect to $\Gamma_1$.  For $q \in \riem_2$, let $J_q(\Gamma)$ be defined using limiting integrals from within $\riem_1$.
  If $h \in \mathcal{D}(\riem_1)$ then $J_q(\Gamma) h = (h,0)$, and if $h \in \mathcal{D}_q(\riem_2)$
 then $J_q(\Gamma) \mathfrak{O}(\riem_2,\riem_1) h = (0,-h)$.
 \end{proposition}
 \begin{proof}  The first claim follows immediately from Theorem \ref{th:Greens_reproducing}.
  The second claim follows from Theorems \ref{th:Greens_reproducing} and  \ref{th:jump_reflection_invariant} (note that $\Gamma$ is negatively oriented with
  respect to $\riem_2$).
 \end{proof}

 We then have a version of the Plemelj-Sokhotski jump formula.
 \begin{corollary}  Let $R$, $\Gamma$, $\riem_1$ and $\riem_2$ be as above.
  Let $H \in \mathcal{H}(\Gamma)$ be such that its extension $h$ to $\mathcal{D}_{\mathrm{harm}}(\riem_1)$
is in $W_1$.   There are unique $h_k \in \mathcal{D}(\riem_k)$, $k=1,2$ so that
  if $H_k$ are their CNT boundary values then $H=- H_2 + H_1$.  These unique $h_i$'s are given by
  \[    h_k = \left. J_q(\Gamma) h \right|_{\riem_k} \]
  for $k=1,2$.
 \end{corollary}
 \begin{proof}
 We claim that $h = - \mathfrak{O}(\riem_2,\riem_1) h_2 + h_1$, which would imply that $H = -H_2 + H_1$.
Proposition \ref{th:jump_dependable} yields
  \[ {\mathfrak{H}} (- \mathfrak{O}(\riem_2,\riem_1) h_2 + h_1) = (h_1,h_2) = {\mathfrak{H}} h. \]
  Thus by Theorem \ref{th:derivative_jump_isomorphism} the claim follows.

  We need only show that the solution is unique.  Given any other solution $(\tilde{h}_1,\tilde{h}_2)$
  we have that $-\mathfrak{O}(\riem_2,\riem_1)(\tilde{h}_2 - h_2) + (\tilde{h}_1 - h_1) \in \mathcal{D}_{\text{harm}}(\riem_1)$ has boundary
  values zero, so by uniqueness of the extension it is zero.  Thus
  \[   0= {\mathfrak{H}}\left( -\mathfrak{O}(\riem_2,\riem_1)(\tilde{h}_2 - h_2) + (\tilde{h}_1 - h_1) \right)
    = (\tilde{h}_1- h_1,\tilde{h}_2 - h_2) \]
  which proves the claim.
 \end{proof}

 Finally, we show that the condition for existence of a jump formula is independent of the choice of side of $\Gamma$.

%More precisely, for $i=1,2$ let $W_1$ and $W_2$ denote the subspace of $\mathcal{D}_{\text{harm}}(\riem_i)$
%   such that
%   \[   \lim_{\epsilon \searrow 0} \int_{\Gamma_\epsilon^i} h(w) \alpha(w) =0   \]
%   for all $\alpha \in A(R)$.  Let $V_i = \overline{A(\riem_i)} \cap \overline{A(R)}^\perp$
%   where it is understood that the orthogonal complement is taken in $A_{\text{harm}}(\riem_i)$.
%   Define
%   \[   V_i' =  \{ \overline{\alpha} + \beta \in A_{\text{harm}}(\riem_i)_e \,:\, \overline{\alpha} \in
%      V_i\}.   \]
 \begin{theorem}  Let $\Gamma$ be a quasicircle and $V_k$, $V_k'$ be as in \emph{Definition \ref{def: V and Vprime}} and $W_k$ as in \emph{Definition \ref{def: W}}.  Then
  \begin{align*}
   \mathfrak{O}(\riem_1,\riem_2) W_1 & = W_2 \\
   \mathfrak{O}_e (\riem_1,\riem_2) V_1' & = V_2'.
  \end{align*}
 \end{theorem}
 \begin{proof}
  The first claim follows immediately from Theorem \ref{th:holomorphic_sides_equal}.   Assume that
  $\overline{\alpha_k} + \beta_k \in A(\riem_k)_e$ for $k=1,2$ are such that
   \[  \mathfrak{O}_e (\riem_1,\riem_2) (\overline{\alpha_1} + \beta_1) =
    \overline{\alpha_2} + \beta_2. \]
  In other words, there are $h_k \in \mathcal{D}_{\text{harm}}(\riem_k)$ such that
  $dh_k = \overline{\alpha_k} + \beta_k$ and $\mathfrak{O}(\riem_1,\riem_2) h_1 = h_2$.  By Stokes'
  theorem, we have that for any $\overline{\alpha} \in \overline{A(R)}$
  \begin{align*}
   \left( \overline{\alpha_k}, \overline{\alpha} \right)_{A_{\text{harm}}}(\riem_k) & =
     \frac{1}{2i} \iint_{\riem_k} \alpha \wedge \overline{\alpha_k} = \frac{1}{2i}
      \iint_{\riem_k} \alpha \wedge dh_k \\
      & = \lim_{\epsilon \searrow 0}  \int_{\Gamma_\epsilon^{p_k}} h_k(w)\, \alpha(w).
  \end{align*}

  Theorem \ref{th:holomorphic_sides_equal} yields that $\overline{\alpha_1} \in V_1$ if and only if
  $\overline{\alpha_2} \in V_2$.
 \end{proof}
\end{subsection}
\end{section}
% \begin{section}{Email correspondence}
% \begin{subsection}{todos}
% \begin{itemize}
% \item change $\mathfrak{G}$ to $\mathfrak{G}$ (genomsnitt)
% \item change $$ to $\mathbf{H}$ (Hleapan= to jump)
% \end{itemize}

% \[   \mathfrak{H}   \ \ \ \ \ \mathbf{H} \ \ \ \ H \ \ \ \  \]

% \end{subsection}
%\end{section}

\end{document}